\pdfoutput=1
\documentclass[paper=letter, fontsize=10pt,leqno]{scrartcl}
\usepackage[body={5.7in,7.5in}]{geometry}
\usepackage[T1]{fontenc}

\usepackage[english]{babel}
\usepackage{csquotes}

\usepackage[protrusion=true,expansion=true]{microtype}				
\usepackage{amsmath}										
\usepackage{amsthm}
\usepackage{latexsym}
\usepackage{hyperref}

\usepackage[pdftex]{graphicx}
\usepackage{url}
\usepackage{amssymb}
\usepackage{bbm}
\usepackage{MnSymbol}

\usepackage[pdftex]{color}
\usepackage{eucal}
\usepackage{amscd}
 \usepackage{verbatim}
 \usepackage{bigints}
 \usepackage[format=plain, font=small]{caption}
\usepackage{epstopdf}
\usepackage{sectsty}					  
\allsectionsfont{\centering \normalfont\scshape}	  

\usepackage{mathtools}

\usepackage{fancyhdr}
\makeatletter
\theoremstyle{definition}
\newtheorem{defn}{\protect\definitionname}
\theoremstyle{plain}
\newtheorem{rmk}{\protect\remarkname}
\theoremstyle{plain}
\newtheorem{prop}{\protect\propositionname}
\theoremstyle{plain}
\newtheorem{assumption}{\protect\assumptionname}
\theoremstyle{plain}
\newtheorem{thm}{\protect\theoremname}
\theoremstyle{plain}
\newtheorem{lem}{\protect\lemmaname}
\theoremstyle{plain}
\newtheorem{cor}{\protect\corollaryname}

\newcommand{\X}{\mathcal{X}}

\newcommand{\inner}[2]{\langle #1, #2 \rangle}

\author{\normalfont \large 
  T. Chumley\footnote{Department of Mathematics and Statistics, Mount Holyoke College, 50 College St, South Hadley, MA 01075}, 
\ R. Feres\footnote{Department of Mathematics and Statistics, Washington University, Campus Box 1146, St. Louis, MO 63130}, 
\ L. Garcia German\footnotemark[2]
}

\makeatother

\providecommand{\assumptionname}{Assumption}
\providecommand{\corollaryname}{Corollary}
\providecommand{\definitionname}{Definition}
\providecommand{\remarkname}{Remark}
\providecommand{\lemmaname}{Lemma}
\providecommand{\propositionname}{Proposition}
\providecommand{\theoremname}{Theorem}

\begin{document}
\title{\textrm{\textmd{\huge{}Knudsen diffusivity in random billiards: spectrum, geometry, and computation}}}
\author{{\large{}T. Chumley}\thanks{Department of Mathematics and Statistics, Mount Holyoke College, 50
College St, South Hadley, MA 01075}, {\large{}R. Feres}\thanks{Department of Mathematics and Statistics, Washington University, Campus
Box 1146, St. Louis, MO 63130}, {\large{}L.~A. Garcia German}\footnotemark[2]}
\maketitle
\begin{center}
Abstract
\par\end{center}
\begin{abstract}
We develop an analytical framework and numerical approach to obtain the coefficient of self-diffusivity for the transport of a rarefied gas in channels in the limit of large Knudsen number.
This framework provides a method for determining the influence of channel surface microstructure on the value of diffusivity that is particularly effective when the microstructure exhibits relatively low roughness. 
This method is based on the observation that the Markov transition (scattering) operator determined by the microstructure, under the condition of weak surface scattering, has a universal form given, up to a multiplicative constant, by the classical Legendre differential operator. 
We also show how characteristic numbers of the system\----namely geometric parameters of the microstructure, the spectral gap of a Markov operator, and the tangential momentum accommodation coefficient of a commonly used model of surface scattering\----are all related.  
Examples of microstructures are investigated to illustrate the relation of these quantities numerically and analytically.  \end{abstract}

\section{Introduction}

\emph{A motivating question and our model.}
In the idealized experiment shown in  Figure \ref{fig:idealized experiment},  a pulse of inert gas
at low pressure is pumped into a long but finite tube, which we refer to as  the {\em channel.} 
The inner surface of the channel has some degree of roughness due to its molecular
structure and surface irregularities. 
The experimenter is able to measure the rate of gas outflow using some device such as a mass spectrometer, 
which generates data of the kind represented by the graph on the right-hand side of the figure.
From such data, transport characteristics of the gas flow through the channel can be derived, as described in \cite{CFZ2016}.
We assume a sufficiently small pulse, under vacuum conditions, to insure  that molecular mean free path is much larger than  the diameter of the channel. 
Thus collisions between the gas molecules can be ignored while gas-surface interaction is expected to influence transport properties most prominently. 
The property  of interest here, which can be indirectly measured from such an experiment, is the Knudsen self-diffusivity coefficient of the gas, as explained, for example, in \cite{CFZ2016}. 
The central question we wish to address is: How do the surface characteristics affect    the Knudsen self-diffusivity?

In this paper, we assume that gas-surface interaction amounts to perfectly elastic, or billiard-like, 
collisions  between  point masses (the {\em gas molecules}, also  referred to here as {\em particles}) and the channel 
surface, and hence  energy exchange between surface and molecules will  be ignored. 
We assume moreover that the channel is two-dimensional and that its surface microstructure is static and periodic, 
and can be described by a relatively small number of geometric parameters.  
Thus the mathematical problem we pose here is to determine how the Knudsen self-diffusivity explicitly depends on these parameters.

In the large Knudsen number limit (i.e., for large mean free paths), 
molecular trajectories are independent of each other 
and the diffusion process is derived from an analysis of individual trajectories of particles undergoing a random flight inside the channel. 
This random flight is governed by a   Markov  operator $P$ that gives, 
at each particle-surface collision, the post-collision velocity of the particle as a random function of the pre-collision velocity. 
 All the information about the periodic surface geometry relevant to the task of obtaining diffusivity is encoded in $P$. 
 In fact, diffusivity corresponds to the variance of a one-dimensional Wiener process   
 obtained from the random flight determined by $P$ via  a Central Limit Theorem. 
  (As explained in \cite{CFZ2016},
  this variance can be obtained from the mean exit time in the limit of long channel lengths in the context of the above idealized experiment.
   The mean exit time, as a function of the channel length, is the only information that needs to be extracted from the exit flow rate data. We won't deal here with this particular aspect of the analysis  and assume, in effect, that the channel is infinite in length.)
   Our main goals are thus centered around two issues. 
   First, we aim to establish a  functional, analytic relationship among the aforementioned variance, the spectrum of the Markov operator $P$, and parameters of the geometric microstructure. 
   Second, we aim to obtain effective numerical methods for finding this dependence for any given geometric microstructure.

   \begin{figure}[h]
\begin{center}
\includegraphics[width=0.8\columnwidth]{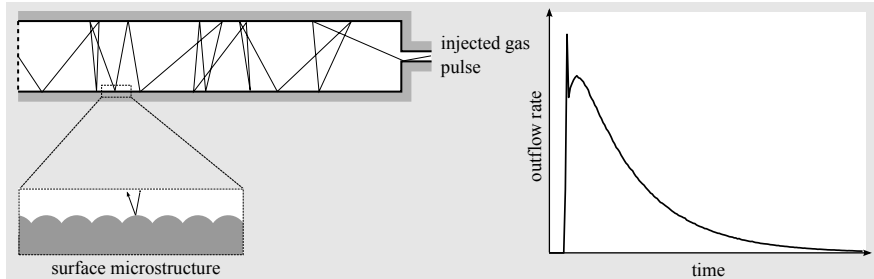}
\caption{ Idealized experiment for measuring diffusivity of a rarefied gas flow through a channel. In the limit of large mean free path, trajectories of gas molecules (point masses) injected into the channel as a short pulse, are independent of each other and their stochastic behavior provides information about the geometric microstructure of the inner surface of the channel. From exit flow rate data one determines the Knudsen self-diffusivity. The mathematical problem posed in this paper is the explicit determination of the diffusivity constant as a function of geometric parameters defining the microstructure.
\label{fig:idealized experiment}}
\end{center}
\end{figure}

\emph{Main results.} 
The main results in the paper are centered around a detailed study of the Markov operator $P$ and
establish analytic and probabilistic properties of $P$ and its corresponding Markov chain. 
To begin, we show that for a large class of microstructures, $P$ has a positive spectral gap,
which in turn establishes the ergodicity of the Markov chain 
as well as the fact that functionals of the Markov chain satisfy the Central Limit Theorem.
We have shown in previous work \cite{Feres2007, FZ2012}
that $P$ is a self-adjoint, compact or quasi-compact operator on an appropriate Hilbert space, 
for microstructures whose sides are concave with curvature bounded away from zero.
However, the present work establishes a positive spectral gap of $P$ for a significantly larger class of microstructures.
Using a conditioning technique, we show that $P$ has positive spectral gap 
when only a certain positive measure portion of the billiard phase space is dispersing.

It is now a classical result in the theory of Markov chains \cite{KV1986} that one can obtain an expression 
of the diffusivity of the Markov chain corresponding to $P$ in terms of an integral over the spectrum of $P$.
A key insight of the present work is that, for relatively flat microstructures, these quantities, 
namely the diffusivity and the spectral gap of $P$, 
are directly connected with a single summary geometric parameter 
that can be computed in a straightforward way from a description of the surface microstructure, 
which we call the surface \emph{flatness parameter} and denote by $h$.
The connection between these three properties of the system is obtained based on a fact which, 
to the best of our knowledge, has not been noted previously in the context of computing the Knudsen diffusivity. 
When scattering by $P$ is relatively weak (in a sense to be made precise), 
it is natural to approximate the operator $P$ in the form $P=I+\mathcal L$, 
where $I$ is the identity operator and $\mathcal L$ may be expected to take the form of a differential (velocity diffusion) operator. 
We show that $\mathcal L$ has a universal form: 
it is a constant (namely, up to a factor of 2, our flatness parameter $h$)
 times the Legendre operator, whose (purely discrete) spectrum is known explicitly. 
We are able to exploit the approximation of $P$ by the Legendre operator
to give an asymptotic expression and error estimates for the Knudsen diffusivity in terms of $h$.

The conceptual link we obtain between $h$, the Knudsen diffusivity, and the spectral gap of $P$ is, in our opinion, 
a new theoretical insight in a very classical subject, which also yields a very effective method of computation,
at least in the case of small values of $h$. 
The final concern of this work is to obtain and validate effective numerical methods for computing the Knudsen  self-diffusivity in terms of the geometric microstructure parameters, in both the small and large $h$ case.
This will be discussed in detail in Subsections \ref{subsec:diff-approx} and \ref{sec:Numerical-techniques-examples} and Sections \ref{sec:Diffusivity} and \ref{sec:Numerical-techniques}. 
A number of numerical experiments involving different microstructures will also be explored. 

\emph{Remarks on assumptions.}
We make a few remarks about the assumptions in our model.
 The analysis developed in this paper does not require in an essential way  all the assumptions made,
 but we hope that the greater simplicity of the present set-up will help to make clearer the main  points.  
For example, we have made a deliberate choice to consider periodic profiles of microstructures with relatively few geometric parameters to emphasize the relationship between geometric parameters, the spectrum of the Markov operator, which in turn establishes the relationship between geometric parameters and Knudsen diffusivity.
While it's possible to give similar formulas relating Knudsen diffusivity and geometric parameters for, say, randomly chosen profiles, such formulas are straightforward but tedious to express, and we fear would muddle the main point. 
 
 When studying the Knudsen self-diffusivity, 
 the observable which measures particle flight between collisions in the channel has infinite variance.
 A study of the Central Limit Theorem and Knudsen diffusivity for a different class of random billiard Markov chains with infinite variance observables has been done previously in \cite{CFZ2016}.
While the methods in the present work can be adapted to the case of infinite variance observables,
we have chosen to use a cut-off observable to reduce to the finite variance case for the sake of clarity.
Besides clarity, there are a number of physically relevant reasons for considering the cut-off observable we have used in our examples.
Namely, (1) the cut-off can arise for macroscopic curvature of the channel in which the test particle traverses. 
It can also arise due to (2) a finite mean free path resulting from unlikely but non-zero probability particle-particle collisions in the large but finite Knudsen number regime, 
and (3) real systems where the channel is of finite length and bounded on either end. 
These physically relevant mechanisms are discussed in detail in \cite{ACM2003}. 

Finally, we should note that in the case of three dimensional cylinders, the inter-collision distance observable is always of finite variance, so our methods in the current paper serve as prototypes for this generalization. 
The techniques we introduce here are in fact not particular to dimension $2$. 
Indeed, a multivariable Legendre operator on the unit disc,
whose spectral theory is explicitly known, 
plays the same role in higher dimensions as the classical Legendre operator does in the present work.
The details of this approximation and the corresponding models in higher dimensions---
where we consider three-dimensional cylindrical channels, parallel plates, and allow for collisions
that induce energy exchange between gas and surface at a given surface temperature---
are at the core of future work currently in preparation.
   
\emph{Related work.}
Better understanding of rarefied gas transport has practical implications for a number of engineering fields including high altitude gas dynamics, porous media, vacuum technology, nano- and microfluidics, among others. These applications have stimulated much experimental work. The following list of   papers is a far from thorough or systematic sample of such  work:  \cite{A2014,M2009,PGEM2011,VNHDV2009,YMN2012}. The reader interested in the more applied side of the subject should consult these sources and others cited in them. From a purely mathematical perspective, this   is a rich source of well motivated and potentially fruitful problems in the general theory of stochastic  processes,  and, more specifically, in 
the study of the stochastic dynamics of {\em random billiard systems}. This is 
our main motivation for studying the subject. We mention from  the mathematical literature  the following, also necessarily incomplete, list: \cite{ABS2013, BBG2019, CPSV2009, E2001, KY2013}.

\emph{Organization of the paper.}
The rest of this paper is organized as follows. In Section \ref{sec:definitions and results} we 
detail our main results after introducing the necessary definitions; 
we define what we call the {\em random billiard Markov chain} model in detail and state
some of its basic properties. Among the main results stated in Section \ref{sec:definitions and results} 
(and proved in more general form later in the paper) we have that
 under certain geometric conditions on the boundary microstructure,
the Markov chain has positive spectral gap and is uniformly ergodic. Numerical
evidence for this is then given for a few  examples. With ergodicity
in hand, we discuss the central limit theory of the Markov chain providing explicit expressions for the variance of the limit diffusion
in terms of the Markov operator $P$. The main analytic technique for computing
diffusivity, based on a Galerkin method for solving a Markov-Poisson equation and a key observation that $P$ is closely related to the Legendre differential operator, is also given in this introductory section.  This approach for obtaining diffusivity is then compared with 
other more straightforward methods for a family of microstructures we call the {\em simple bumps} family.
A few more examples of microstructures are explored, having in mind the  relation between geometric parameters, diffusivity, and spectral gap. 
Section \ref{sec:Spectral-gap} is dedicated to stating and proving the analytical results of the paper in their general form, 
while Section \ref{sec:Diffusivity} details, and adds further information, to the numerical methods and their validation. 

\section{Main definitions and results\label{sec:definitions and results}}
\subsection{The billiard cell and its transition operator $P$}
The notation $\mathcal{P}(\Omega)$ will be used below to denote the space of probability measures on a measurable space $\Omega$. If $\mu$
is a measure on $\Omega$ and $f:\Omega\rightarrow \mathbb{R}$ is $\mu$-integrable, we write the integral of $f$ with respect to $\mu$ as
$$\mu(f):=\int_\Omega f(\omega)\, \mu(d\omega). $$
The Hilbert space of square integrable functions  with respect to $\mu$ 
and its subspace of functions with mean zero will be written
$$L^2(\Omega,\mu)=\left\{f: \mu\left(f^2\right)<\infty\right\},  \ L_0^2(\Omega,\mu)=\left\{f\in L^2(\Omega,\mu): \mu(f)=0\right\},$$   with  inner product    $\langle f,g\rangle_\mu:=\int_\Omega f(\omega)g(\omega)\, \mu(d\omega)$ and  norm $\|f\|_\mu:=\langle f,f\rangle_\mu^{1/2}$.
Moreover, we define a norm on the space of square integrable probability measures on $\Omega$ which are absolutely continuous with respect to $\mu$ as follows. Let $\nu$ be such a measure, so that $f$ is the Radon-Nikodym derivative of $\nu$ with respect to $\mu$. Then $\|\nu\|_\mu := \|f\|_\mu$.

 \begin{figure}[h]
\begin{centering}
\includegraphics[width=0.8\columnwidth]{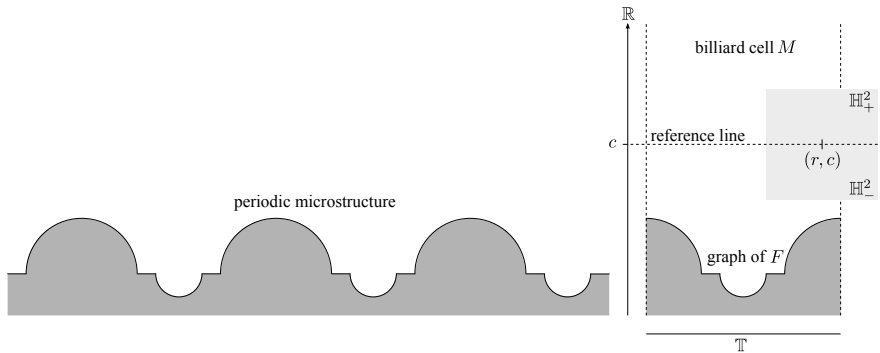}
\par\end{centering}
\caption{ A periodic microstructure and its billiard cell, with some of the notation used  to define the random billiard map and  its transition operator $P$.
For some of our results we assume that the boundary curve is the graph of  a piecewise smooth function $F:\mathbb{T}\rightarrow \mathbb{R}$.
\label{fig:billiard cell}}
\end{figure}

The general set-up will be that of a two-dimensional {\em random  billiard} with static, periodic, geometric microstructure, as in \cite{CFZ2016,CF2012,Feres2007,FNZ2013,FY2004,FZ2012}.  The periodic structure 
is defined by the choice of a {\em billiard cell} $M$, from which the Markov operator $P$ will be defined.  The billiard cell is a subset $M$ of  $\mathbb{T}\times \mathbb{R}$, where $\mathbb{T}$ denotes the $1$-dimensional torus (equivalently, the interval $(0,\ell)$ with periodic condition imposed at the endpoints, where $\ell$ will typically be set equal to $1$.)  
The boundary of the billiard cell is assumed to be a piecewise smooth curve.
For some of the results given below, the boundary  will be the graph of a piecewise smooth function $F:\mathbb{T}\rightarrow \mathbb{R}$, so that $M$ consists of the points $(r,y)$ such that $y\geq F(r)$. Choose an arbitrary value $c$ such that $c>F(r)$ for all $r\in \mathbb{T}$. The line $y=c$ will  be called the {\em reference line}. At any point $(r,c)$ on the reference line we define the half spaces $\mathbb{H}^2_-$ and
$\mathbb{H}^2_+$ of incoming and outgoing velocities, respectively. Thus $(r, c, v)\in M\times \mathbb{H}^2_-$ represents the initial conditions of 
an incoming particle trajectory. These conditions uniquely specify (for almost every $r$ and $v$) a billiard trajectory: upon hitting a non-corner point on the cell boundary, the particle reflects specularly without changing speed, and upon crossing a vertical boundary line of $M$ (more precisely, a line
separating two adjacent cells, represented 
 in Figure \ref{fig:billiard cell} by the vertical dashed lines) it reenters the other (dashed) line with unchanged  velocity.  With probability $1$ on the set of initial conditions (due to Poincar\'e's recurrence), the trajectory returns to the reference line, at which point we register its   outgoing velocity
 $V(r,v)\in \mathbb{H}^2_+$ and new position $r'$. Without risk of confusion we may identify (through reflection about the reference line) $\mathbb{H}^2_-$ and $\mathbb{H}^2_+$, denoting both by
 $\mathbb{H}^2$. We have thus defined a transformation $(r,v)\mapsto (r',V(r,v))$ (for almost all initial conditions $(r,v)$) on $\mathbb{T}\times \mathbb{H}^2$.    We call   this transformation the  {\em return billiard map}.

 Note that the vector norms satisfy $|v|=|V|$ since collisions are elastic. We may, without loss of generality, assume that the particle trajectories have unit speed.  The incoming or outgoing {\em state space}, consisting of  initial or return velocities, can then be taken to be the interval $\mathcal{X}=(0,\pi)$ of angles the particle velocity makes with  the reference line.  We can (and often will) equivalently define $\mathcal{X}=(-1,1)$ as the set of values of the  cosine of those angles. Given an initial velocity $x \in \mathcal X$, we will often denote the return velocity by $X(r,x)\in\mathcal X$ in analogy with the earlier notation of velocities $v$ and $V(r,v)$ in $\mathbb{H}^2$.
 
 Let $\mathcal{P}(\mathcal{X})$ denote the space of probability measures on $\mathcal{X}$. Given an incoming velocity $v$, let us suppose that $r=U$
 is a random variable with the uniform distribution over $\mathbb T$. Thus $X(U,x)$ becomes a random variable.  We now define the Markov (or transition probabilities) operator $P$ as follows.
 Let $f$ be any bounded and continuous function on $\mathcal{X}$ and define
 $$\left(Pf\right)(x):=E\left[f(X(U,x))\right] = \int_\mathbb{T} f(X(r,x))\, dr, $$
 where $dr$ is the length element of normalized Lebesgue measure on $\mathbb T$.
 Equivalently, we define a sequence of random variables $(X_n)_{n\geq 0}$ with a given initial distribution $\mu \in \mathcal P(\mathcal X)$ as follows. Let $(U_n)_{n\geq0}$ be an independent, identically distributed sequence of random variables uniformly distributed on $\mathbb T$, and, for each $n\geq 0$, let
 $$X_{n+1} := X(U_n, X_n).$$
    The justification for assuming, at each scattering event, that the point $r$ of entry over the opening of a billiard cell is random and uniformly distributed is due to our regarding the billiard cell as being very small relative to other length scales; any small uncertainty in the incoming velocity will make $r$ nearly fully uncertain. See \cite{FY2004} for a more detailed explanation of this point. 
 
We can also regard $P$ as a map from $\mathcal{P}(\mathcal{X})$ to itself:  Given any $\mu\in\mathcal{P}(\mathcal{X})$, let  $\mu P\in \mathcal{P}(\mathcal{X})$ be such that for any test function $f$ (bounded and continuous),
$$(\mu P)(f) := \mu(Pf).$$

 The following   summarizes the   basic properties of $P$.
For  
 their proofs, see \cite{Feres2007,FZ2012}. We say that the billiard cell $M$ is {\em bilaterally symmetric} (or simply {\em symmetric}) if it is invariant under reflection through the middle vertical
 line. When the boundary of the cell is the graph of a function $F$, this means that
 $F(r)=F(\ell-r)$
 for all $r\in (0,\ell)$.
 
\begin{prop}\label{prop:P}
The Markov operator $P$, for any given billiard cell,
has the following properties.
\begin{enumerate}
\item The measure $\pi\in\mathcal{P}(\mathcal{X})$ given by $\pi(d\theta)=1/2\sin\theta\,d\theta$
is stationary of $P$. That is, $\pi P=\pi$. 
\item As an operator on $L^{2}(\mathcal{X},\pi)$, $P$ has norm 1.
\item If  $M$   is  symmetric,  $P$ is self-adjoint
  and the stationary 
Markov chain is reversible.
\end{enumerate}
\end{prop}
 
Note that when $\mathcal X = (-1,1)$, it is straightforward to see by a change of variables that the stationary measure $\pi$ is given by the uniform measure $\pi(dx) = 1/2\,dx$.

 If the the billiard cell is not bilaterally symmetric, the adjoint of $P$ is  still closely related to $P$ as described in \cite{CF2012} and much of the analysis developed in this paper still applies. For simplicity, we do not consider the more general type of cells here. 

\subsection{Spectral gap and ergodicity\label{sec:Spectral-gap}}

Let $(X_n)_{n\geq 0}$ be the Markov chain with transition operator $P$ and initial distribution $\mu$. Then the measure $\mu P^n$ is the law
of the $n$th step $X_n$.  We are interested in the convergence of $\mu P^n$ to the stationary measure $\pi$ in the sense of total variation. Recall that the {\em total variation} of a measure $\mu$ is defined as
$$\|\mu\|_{v}:= \sup_{A\subset \mathcal{X}} |\mu(A)|. $$
\begin{defn}
A Markov chain with stationary distribution $\pi$ is \emph{$\pi$-a.e.
geometrically ergodic }if there exists $0<\rho<1$ such that for $\pi$-a.e.
$x\in\mathcal{X}$ there exists a constant $M_{x}>0$ possibly dependent
on $x$ such that
$
\left\Vert \delta_{x}P^{n}-\pi\right\Vert _{v}\leq M_{x}\rho^{n}
$
for all $n\geq1$.
\end{defn}

The operator
$P$ has \emph{spectral gap }if there exists a constant $0<\rho<1$
such that $$\left\Vert Pf\right\Vert _{\pi}\leq\rho\left\Vert f\right\Vert _{\pi}$$
for all $f\in L_{0}^{2}(\mathcal{X},\pi)$. The value $\gamma:=1-\rho$
is called the {\em spectral gap} of $P$. It is straightforward
to see that  for a compact and self-adjoint $P$, $\rho$ is     given by the largest
eigenvalue of $P$ restricted to $L_{0}^{2}(\mathcal{X},\pi)$ and
$\gamma>0$. Finally, we note that if $P$ has spectral gap and is
self-adjoint, then for any initial distribution $\mu$ which is absolutely
continuous with respect to $\pi$, there exists a constant $M_{\mu}>0$
such that
\[
\left\Vert \mu P^{n}-\pi\right\Vert _{v}\leq M_{\mu}\rho^{n}.
\]
See \cite{RR1997}. We will prove
geometric ergodicity for a large class of microstructures satisfying
certain geometric conditions.

\begin{figure}[h]
\begin{centering}
\includegraphics[width=0.3\columnwidth]{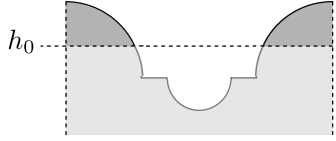}
\par\end{centering}
\caption{An example of billiard cell for which Theorems \ref{thm:geo_erg} and \ref{thm:clt} hold. Other than being bilaterally symmetric, its shape is essentially arbitrary below a line $y=h_0$ whereas above it, the boundary  consists of two smooth concave lines with  curvature bounded below by some positive number $K$.}
\label{fig:theorem hypothesis}
\end{figure}

The following   is a special case of a more general result to be stated  and proved in Section \ref{sec:Spectral-gap}.   We call the {\em height of the billiard cell} the supremum of the $y$ coordinate function restricted to  the boundary of the cell. 

\begin{thm}\label{thm:geo_erg}
Let $P$ be the Markov transition operator  for a random billiard Markov chain whose billiard cell is  symmetric and satisfies the following property:
above a certain   $y=h_0$   strictly less than the height of the cell, the cell boundary  is the union of smooth, concave curves having  curvature bounded away from $0$. Then $P$ is a self-adjoint   operator  with a positive spectral gap. As a result, there exists a constant $\rho\in (0,1)$ such that for each  $\mu \in \mathcal{P}(\mathcal{X})$ with $\|\mu\|_\pi<\infty$,  
$$ \left\|   \mu P^n -\pi \right\|_v\leq M_\mu \rho^n$$
for some constant 
$M_\mu<\infty$ and $n\geq 1$.
\end{thm}

Figure \ref{fig:theorem hypothesis} gives an example of billiard cell for which Theorem \ref{thm:geo_erg} holds.

\subsection{Central Limit and Diffusivity}
Referring back to Figure \ref{fig:idealized experiment}, one expects for a sufficiently long  channel that the molecular random flight 
can be approximated by a Wiener process whose variance corresponds to the Knudsen self-diffusivity. This is justified by a Central Limit Theorem (CLT). This diffusivity has a convenient expression  when the   transition operator $P$ is   self-adjoint.  We
describe this expression here and prove further details later in the paper.

 Let $(X_n)_{n\geq 1}$ be, as above, the stationary Markov chain generated by $P$, with stationary probability measure $\pi$. Recall that $X_n$ has values in the space of post-collision velocities $\mathcal{X}$. This space can be parametrized by the values of the cosine of the angle the velocity vector makes with the horizontal reference line $y=c$. 
 (See Figure \ref{fig:billiard cell}.) Thus we may set $\mathcal{X}=(-1,1)$. Let $f:\mathcal{X}\rightarrow \mathbb{R}$ be the observable 
 $$\tilde f(x)= 2 r x\left(1-x^2\right)^{-1/2} $$
where $r$ is the radius of the channel. We suppose, in the context of formulating a CLT for  molecular trajectories, that the length
of the channel is infinite. Note that $\tilde f(X_n)$ is the distance travelled by the particle along the channel's horizontal axis between the $n$th and the  $n+1$st collisions with the channel wall.
The total horizontal displacement up to the $n$th collision is
$$S_n(\tilde f)=\sum_{k=0}^{n-1} \tilde f(X_k).$$
In its standard form, the CLT gives a limit in distribution for expressions of the form $S_n(f)/\sqrt{n}$ where $f$ is an observable having mean zero and finite variance. A simple calculation shows that the horizontal displacement function $\tilde f$ has mean zero but infinite variance. For this reason we consider instead the following modified, cut-off displacement observable:
\begin{equation}\label{eqn:truncated}f_a(x):= \tilde f(x) \mathbbm{1}_{\{|\tilde f|\leq a\}}(x)  + a \mathbbm{1}_{\{|\tilde f|>a\}}(x)\end{equation}
for large $a>0$. Here $\mathbbm{1}_I(x)$ denotes the indicator
function of the set $I$, which is defined as $\mathbbm{1}(x)=1$ if $x\in I$ and $0$ if $x\notin I$.
There are a number of physical mechanisms  that could be invoked to make this cut-off plausible.  For example, the channel might have
a slight curvature along its length, setting an upper bound on the horizontal distance traveled. See \cite{ACM2003}  for an outline of other mechanisms.
We should also note that while the CLT with the
usual scaling does not hold for the observable $\tilde f$, the distribution of $\tilde f(X_n)$ is still
in the domain of attraction of the Gaussian law. One can check
that $\tilde f$ is slowly varying and, as a result, a CLT
with nonstandard scaling holds for random billiard Markov chains with
sufficient mixing. See \cite{CFZ2016} for a detailed study of such
Markov chains. The program we outline in this paper   to estimate the diffusivity should hold in the
infinite variance case as well, but we have chosen to focus on the
finite variance case for the sake of clarity of exposition. It should
also be noted that for cylindrical  channels in dimension 3 (and higher),
the observable that gives the distance traveled along the axis of
the channel is of finite variance. 

We suppose the microstructure satisfies the same geometric assumptions of Theorem \ref{thm:geo_erg}. In particular, $P$ is self-adjoint and has positive spectral gap. 
  Let $\Pi$ be the {\em spectral resolution} of $P$\----the projection-valued
 measure on the spectrum $\sigma(P)\subset [-1,1]$ granted by the Spectral Theorem for bounded self-adjoint operators.  Then
 $$ P=\int_{-1}^1 \lambda\, \Pi(d\lambda).$$
Let $f$ be any observable in $L_0^2(\mathcal{X}, \pi)$ (for example,  the truncated displacement function $f_a$) and define the measure
$\Pi_f$ supported on $\sigma(P)\setminus \{1\}$ by $$\Pi_f(d\lambda):=\langle f, \Pi(d\lambda)f\rangle_\pi.$$

The following is a special case of a theorem that will  be stated and proved in Section \ref{sec:Diffusivity}.
 
\begin{thm}\label{thm:clt}
Let $(X_n)_{n\geq 0}$ be a Markov chain taking values in $\mathcal{X}$ with Markov transition operator $P$ and stationary measure $\pi$. 
Suppose $P$ is associated to a billiard cell satisfying the same geometric assumptions of Theorem \ref{thm:geo_erg}.
Let $f\in L^2_0(\mathcal{X},\pi)$. Then $S_n(f)/\sqrt{n}$ converges in distribution  to a centered Gaussian random variable $\mathcal{N}(0, \sigma_f^2)$, where the variance is given by
$$\sigma_f^2 = \int_{-1}^1 \frac{1+\lambda}{1-\lambda} \Pi_f(d\lambda) = \langle f, f\rangle_\pi + 2\langle f, P(I-P)^{-1}f\rangle_\pi. $$  
\end{thm}

The expression for the diffusivity given above
 suggests the following approach for computing $\sigma_{f}^{2}$. 
 Let $L:=P-I$ be the \emph{Markov Laplacian} and $g$ the solution to the 
  \emph{Markov-Poisson equation $Lg=-f$}. Then 
  the dimensionless Knudsen self-diffusivity coefficient takes the form
\begin{equation}\label{eqn:eta} \eta=\frac{\sigma^2_f}{\sigma_0^2} = 1 +2 \|f\|_\pi^{-2} \left\langle f, Pg\right\rangle_\pi,\end{equation}
 where $\sigma_0^2=\|f\|_\pi^2$ is the diffusivity for the  process with independent post-collision velocities with  the identical distribution $\pi$. 
 In the next subsection we explain one approach to carrying out this program by approximating $L$ by an elliptic differential operator
 $\mathcal{L}$ whose spectral theory is well
understood.
  It turns out that  $\mathcal{L}$  has a canonical form as we show next.

\subsection{The Legendre Equation and Diffusion Approximation}\label{subsec:diff-approx}
Our aim now is to show that it is possible to approximate the solution
of the Markov-Poisson equation $Lg=-f$ for a large class of random
billiard microstructures when  $P$ is close to the identity operator $I$. We consider families of microstructures
indexed by a scalar quantity $h$ that, in a sense to be made precise,
characterizes a key geometric feature of the microscopic billiard
cell, namely its flatness. For each microstructure with parameter
$h$, the corresponding Markov operator $P_{h}$ defines the dynamics
of the random billiard Markov chain as discussed previously. The key
idea now is that for small values $h$, the operator $P_{h}$ will
act nearly like the identity operator, due to the flatness of the
geometry; the  Markov-Laplace operator $L_{h}:=P_{h}-I$,
in the limit as $h\to0$ and under some general assumptions
on the microscopic billiard cell,   will then have  a canonical approximation   by
the classical Legendre differential operator, whose spectral theory is well understood.
In the rest of the subsection, we make explicit the necessary assumptions
on the geometry and give the statement  of our operator approximation
result and provide examples. 

Let the boundary of the  billiard cell be   the graph
of a periodic function $F:\mathbb{T}\to\mathbb{R}.$ (See Figure \ref{fig:billiard cell}.)
In order to characterize how flat the microstructure boundary is,
we consider the normal vector field $\mathbbm n:\mathbb{T}\to\mathbb{R}^{2}$
along the graph of $F$, and let $\bar{\mathbbm n} = \bar{\mathbbm n}(r)$ denote its projection
onto its first (horizontal) component. Finally, we let
\begin{equation}  
h:=\int_{\mathbb{T}}\bar{\mathbbm n}^{2}\,dr = \int_{\mathbb T}\frac{F'(r)^{2}}{1+F'(r)^{2}}\,dr.\label{eq:flat}
\end{equation}
It will be seen in examples that   $h$ captures information
about the curvature of the boundary. For small values of $h$, the collision events
with the boundary will be relatively simple, often resulting in only
a single collision with the cell's boundary and only  
a small deviation from specular reflection. This implies little change in the tangential momentum of the particle with high probability. It is in this sense that $h$ 
can be thought to have a  role similar to the accommodation coefficient $\vartheta$ referred to earlier in the paper.

 Let $\mathcal X = (-1,1)$ and let $\mathcal{L}$ denote the differential operator acting on smooth functions
$f:\mathcal{X}\to\mathbb{R}$ as
\begin{equation}
\mathcal{L}f(x)=\frac{d}{dx}\left(\left(1-x^{2}\right)\frac{d}{dx}f(x)\right).\label{eq:legendre}
\end{equation}

\begin{thm}
\label{thm:diffusion-app}Let $(F_{h})_{h>0}$ be a
family of piecewise smooth functions $F_{h}:\mathbb{T}\to\mathbb{R}$ defining bilaterally symmetric billiard cells, indexed by the
flatness parameter $h$   introduced in (\ref{eq:flat}). 
Let $(P_{h})_{h>0}$ be the corresponding Markov transition operators.
Then for any $f\in C^{3}(\mathcal{X})$,   
\[
L_{h}f(x)=2h\mathcal{L}f(x)+O\left(h^{3/2}\right)
\]
holds   for each $v$ such that every initial condition with velocity $v$ results in a trajectory that collides only once with the boundary of the cell.
\end{thm}

In the context of Theorem \ref{thm:diffusion-app} we observe that, for each $x \in \mathcal X$, every initial condition with velocity $x$ results in a trajectory 
that collides only once  cell boundary as long as we take
 $h$ to be sufficiently small.

The differential operator $\mathcal{L}$ has a well understood spectral
theory that will be used to obtain information about
$P_{h}$. We recall that the eigenvalue problem $\mathcal{L}f=\lambda f$ has square integrable solutions  
if and only if $\lambda$ is of the form $\lambda=-l(l+1)$ for integers
$l\geq0$. 
The associated eigenfunctions are the Legendre polynomials $\phi_{l}$,
$l\geq0$, 
 $$\phi_{0}=1,\  \  \phi_{1}(x)=x,\ \ \phi_{2}(x)=(3x^{2}-1)/2, \dots.$$
  The collection $(\phi_{l})_{l\geq0}$  forms
a complete orthogonal basis for $L^{2}(\mathcal X, \pi)$ and
\[
\langle\phi_{n},\phi_{m}\rangle_\pi =\int_{\mathcal X}\phi_{n}(x)\phi_{m}(x)\,\pi(dx)=\frac{1}{2n+1}\delta_{n,m},
\]
where $\delta_{n,m}$ is the Kronecker delta symbol.

 As a
first application of the approximation given in Theorem \ref{thm:diffusion-app},   
we give an informal estimation of the spectral gap $\gamma_h$ of $P_{h}$
for values of $h$ near 0. 
Note that the largest eigenvalue of $P_{h}$
is 1, with eigenfunctions given by the constant functions. So   $\gamma_h$ is given by $1-\lambda$ where $\lambda$
is the second largest eigenvalue of $P_{h}$. Using the approximation
in Theorem \ref{thm:diffusion-app},
\[
P_{h}\phi_{l}=\left(1-2hl(l+1)\right)\phi_{l}+O(h^{3/2}),
\]
where $\phi_{l}$ is the Legendre polynomial associated to eigenvalue
$-l(l+1)$. This suggests that the second largest eigenvalue $\lambda$
of $P_{h}$ is given by
$
\lambda\approx1-4h
$. Equivalently, this suggests the following asymptotic  estimate of $\gamma_h$: 
\begin{equation}\label{eq:gap}\gamma_h\approx  4h. \end{equation}

  The idea then will be to use the  
approximation $\mathcal{L}$ of the Markov-Laplacian $L$ in order
to give an approximation of the function $g=(I-P)^{-1}f$ that appears
in the equation $$\sigma_f^2= \left\langle f, f\rangle_\pi + 2\langle f, P(I-P)^{-1}f\right\rangle_\pi$$
obtained in Theorem \ref{thm:clt}.
 Note that $g$ is a solution
of the Markov-Poisson equation $Lg=-f$. The following thorem  shows that a series
solution of the  Poisson equation  for $\mathcal{L}$ can be given explicitly
in terms of Legendre polynomials.

\begin{thm}\label{thm:sigma_approx}
Let $(P_{h})_{h>0}$ be a family of random billiard Markov transition
operators for a family of   billiard cells satisfying the geometric assumptions of Theorems \ref{thm:clt} and \ref{thm:diffusion-app}. For any function $f\in L_{0}^{2}(\mathcal{X},\pi)$,
let $\sigma_{f,h}^{2}$ denote the diffusivity corresponding to $P_{h}$.
Then
\begin{equation}
\sigma_{f,h}^{2}=-\langle f,f\rangle_\pi+\frac{1}{h}\sum_{l=1}^{\infty}\frac{2l+1}{l(l+1)}\left\langle \phi_{l},f\right\rangle_\pi^{2}+O(h^{1/2}).\label{eq:variance-h-approx}
\end{equation}
\end{thm}

\begin{rmk}
It should be noted that for the sake of numerical computations, it is natural to consider the quantity given by truncating the series in (\ref{eq:variance-h-approx}) after a fixed number of terms $n \geq 1$, so that
$$\sigma^2_{f,h} = -\langle f, f \rangle_\pi + \frac 1h \sum_{l=1}^n \frac{2l+1}{l(l+1)} \langle \phi_l, f \rangle_\pi^2 + E_{h,n},$$
where $E_{h,n}$ is the tail of the series along with the $O(h^{1/2})$ error term. This quantity can be estimated as follows:
$$E_{h,n} = \frac 1h \sum_{l=n+1}^\infty \frac{2l+1}{l(l+1)}\langle \phi_l, f \rangle_\pi^2 +O(h^{1/2}) \leq \frac{\|f\|^2_\pi}{h} \sum_{l=n+1}^\infty \frac{2l+1}{l(l+1)}\|\phi_l\|_\pi^2 + O(h^{1/2})= \frac{\|f\|_\pi^2}{h(n+1)} + O(h^{1/2}).$$
\end{rmk}

The theorem implies that the dimensionless self-diffusivity coefficient satsifies
$$\eta_f = -1 +  \frac{1}{h}\sum_{l=1}^{\infty}\frac{2l+1}{l(l+1)}\left\langle \phi_{l},f/\|f\|_\pi\right\rangle_\pi^{2}+O(h^{1/2}) = -1 + \frac1h C_f + O\left(h^{1/2}\right),$$
where $C_f$ is defined by this identity. Thus, for small $h$,
\begin{equation}\label{eq:approx} \eta_f \approx \frac{C_f -h}{h}.\end{equation}
Then the approximate identity  (\ref{eq:gap}) 
suggests
\begin{equation}\label{eq:approx_gamma} \eta_f \approx \frac{4C_f -\gamma}{\gamma}.\end{equation}
It is interesting to compare this expression with the one obtained under the Maxwell-Smoluchowski model:
$$\eta = \frac{2-\vartheta}{\vartheta}$$
where $\vartheta$ is the accommodation coefficient, defined as the fraction of diffuse collisions. 
We thus obtain a conceptual relation linking the purely geometric quantity $h$ (flatness), the spectral quantity $\gamma$ (spectral gap), and
the tangential momentum accommodation coefficient $\vartheta$ defined for a standard and widely used collision model.
Finally, it
is worth comparing these expressions with the exact equation
$$ \eta_f = \int_{-1}^1 \frac{2-\vartheta}{\vartheta} \overline{\Pi}_f(d\vartheta)$$
where $\overline{\Pi}_f(d\vartheta)=\Pi_f(d\vartheta)/\|f\|^2_\pi$, which is obtained from
Theorem \ref{thm:clt}  by setting $\vartheta=1-\lambda$.

\subsection{Two Examples}
Consider the microscopic billiard cell, which we will refer to as
the {\em small bumps} microstructure throughout the discussion, whose boundary
is given by arcs of circles as in Figure \ref{fig:bumps}. The
geometric parameter of interest here is the dimensionless curvature
given by $K=\ell/R$, where $R$ is the radius of one of the arcs and $\ell$ is the length of the opening to the billiard cell as shown in the figure. 
An elementary computation using (\ref{eq:flat}) gives
\[
h=\frac{K^{2}}{12}.
\]
As a result, the spectral gap, approximated for values of $K$ near
zero, is given by $$1-\lambda\approx4h=K^{2}/3.$$  
Figure \ref{fig:bumps-spectral-gap} shows the numerically obtained values for the spectral gap and $\eta$ compared to the respective approximations as functions of the dimensionless curvature parameter $K$.

\begin{figure}[h]
\begin{centering}
\includegraphics[width=0.4\columnwidth]{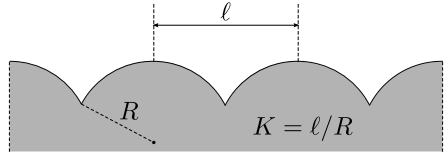}
\par\end{centering}
\caption{The {\em bumps} microstructure with dimensionless curvature parameter $K$.}
\label{fig:bumps}
\end{figure}

\begin{figure}[h]
\begin{centering}
\includegraphics[width=1.0\columnwidth]{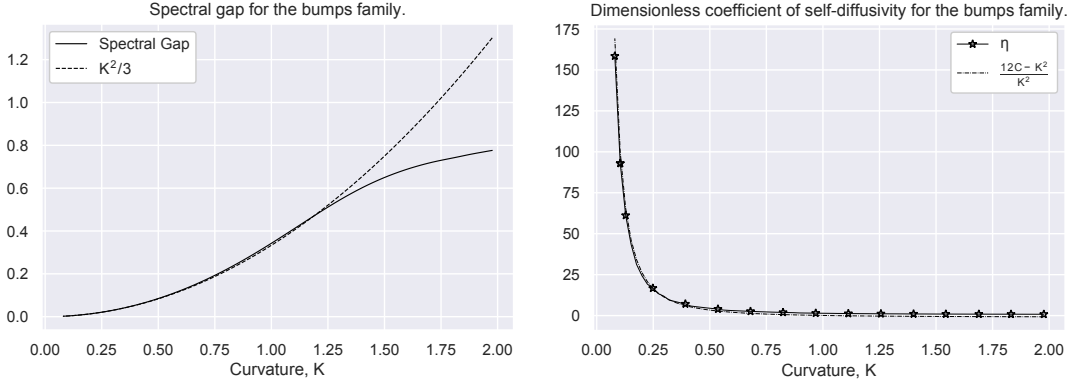}
\par\end{centering}
\caption{Left: the spectral gap of the operator $P$ for the bumps  family of microstructures depicted in Figure \ref{fig:bumps},
with dimensionless curvature parameter $K$, compared with the
approximation of the Markov-Laplacian by the Legendre differential operator. The solid
curve is constructed from the numerical approximation detailed in
Section \ref{sec:Numerical-techniques-examples}. Right: comparison of the dimensionless diffusivity coefficient $\eta$ obtained using  (\ref{eqn:eta})
and a finite dimensional approximation of $P$ (indicated on the graph by the stars) and the approximation of $\eta$ as a  function of the geometric parameter
given by (\ref{eq:approx}). The observable is $f_a$ with cut-off $a=50000$. 
\label{fig:bumps-spectral-gap}}
\end{figure}

A similar computation can be done for the microgeometry in Figure
\ref{fig:bumps alpha} that consists of a mixture of the small bumps
geometry together with flat, specularly reflecting lines. In this
case, the family is parameterized by the proportion of initial positions
$\alpha$ that result in reflections with the part of the boundary
with curvature. After expressing the boundary as the graph of an appropriately
defined function and computing an elementary integral, we get that
$h=\alpha/3$. 
 
 Generalizing this second example,  consider  the transition operator $$P_\alpha = \alpha P_1+(1-\alpha)I $$ 
 where $P_1$ is the operator associated to a given microstructure. 
 Then $P_\alpha$ is associated to the microstructure for which a   segment of horizontal line of length $d$ is added to the billiard cell of the first microstructure.
 The parameter $\alpha$ is then the probability that an incoming particle will not collide with the flat segment. It is easy to see the effect of the additional parameter $\alpha$. Note that $P_\alpha-I = \alpha(P_1-I)$. An elementary algebraic manipulation starting from the expression
 $$ \sigma_{f,\alpha}^{2} =\left\langle f,f\right\rangle_\pi +2\left\langle P_{\alpha}f,(I-P_{\alpha})^{-1}f\right\rangle_\pi$$
gives
$$\eta_{f,\alpha} = \eta_{f,1} + \frac{2(1-\alpha)}{\alpha} \frac{\langle f, (I-P_1)^{-1}f\rangle_\pi}{\|f\|_\pi^2}$$
where $f$ is  arbitrary. As it is to be expected, $\eta_{f,\alpha}$ approaches infinity as the probability of specular reflection increases to $1$.

\begin{figure}[h]
\begin{centering}
\includegraphics[width=0.5\columnwidth]{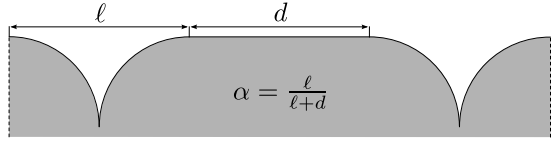}
\par\end{centering}
\caption{Adding a flat segment to a given microstructure, as indicated in this diagram, gives the transition operator $P_\alpha=\alpha P_1+(1-\alpha)I$,
where $P_1$ is the operator associated to the original microstructure. }
\label{fig:bumps alpha}
\end{figure}

\subsection{Summary of the numerical techniques and examples \label{sec:Numerical-techniques-examples}}

In equation (\ref{eq:variance-h-approx}) of Theorem \ref{thm:sigma_approx}, we have given our main numerical approach of the paper with respect to analyzing the regime of small flatness parameter $h$; namely, we estimate the dimensionless self-diffusivity $\eta = \eta_f$ by truncating the series in equation (\ref{eq:variance-h-approx}).
In this subsection we outline an additional numerical approach for computing the dimensionless self-diffusivity $\eta$ (or, equivalently, the variance $\sigma_f^2$ of the Gaussian limit of the random flight in a channel).  
This method, which we will refer to as the Galerkin method, requires us to introduce a finite rank approximation,
or discretization, of the Markov operator $P$, which we describe below.
The purpose of introducing this additional approach is two-fold. 
First, the Galerkin method serves as numerical verification of the main approach of using equation  (\ref{eq:variance-h-approx}). 
Additionally, the Galerkin method is applicable to microstructures which fall outside of the small $h$ regime.
As we will see, however, the method has the disadvantage of requiring a discretization of $P$ and, for this reason, is more computationally demanding.
We conclude this subsection with a discussion of some additional examples that show the subtle relationship between the spectral gap, the dimensionless self-diffusivity, and geometric features of the microstructure.

The starting point in computing $\eta$ is the equation 
$\sigma_f^2= \left\langle f, f\rangle_\pi + 2\langle f, P(I-P)^{-1}f\right\rangle_\pi,$
which in turn requires that we obtain the solution $g$ to the Markov-Poisson equation $(P-I)g=-f.$
The classical Galerkin method
gives us a general approach for solving this equation as follows
(see \cite{atkinson1987discrete} for a broader discussion of the approach). 
For each $n \geq 1$, let $T_n:L_0^2(\mathcal{X},\pi)\rightarrow R_n$ denote the orthogonal projection to 
the linear span $R_n=\{\phi_1, \dots, \phi_n\}$ of Legendre polynomials defined on $\mathcal{X}$.
Define $g_n\in L_0^2(\mathcal{X},\pi)$ to be the solution of the finite dimensional linear system
	$$ (I-T_nP)g_n = T_nf.$$
Equivalently, we find $g_n\in R_n$ so that 
$\left\langle (I-P)g_n, \psi\right\rangle_\pi = \langle f, \psi\rangle_\pi$
for all $\psi\in R_n$, which can be done as follows.  
Writing  $g_n=\sum_{j=1}^n \alpha_j \phi_j$ 
and defining $x=(\alpha_1, \dots, \alpha_n)^\intercal, 
y= (\langle f, \phi_1\rangle_\pi, \dots, \langle f, \phi_n\rangle_\pi)^\intercal,$ 
and $G=\left(\langle \phi_j,\phi_i\rangle_\pi - \langle P\phi_j, \phi_i\rangle_\pi\right)_{i,j=1}^n,$
we are left to do two computations.
First, we find the entries of the matrix $G$.
Second, we find the solution $x$ to the linear system $Gx=y$. 
This then gives the solution $g_n$ to the finite dimensional linear system, and from it the approximate value $\sigma^2_{\mathrm{GM},n}$.
The following theorem provides an error estimate for this approximation.
A proof is given in Section \ref{sec:Numerical-techniques}. 
Figure \ref{fig:Galerkin error} gives numerical verification of the  convergence and error bound for $\sigma^2_{\mathrm{GM},n}$ 
as given in the theorem.

\begin{thm}\label{thm:galerkin} 
Let $f\in L_0^2(\mathcal{X},\pi)$, where $\mathcal{X}=(-1,1)$, be such that the first derivative $f'$ is absolutely continuous and the second derivative $f''$ is of bounded variation.  Let $\sigma_f^2$ be defined by the equation
$$\sigma_f^2 = \langle f, f\rangle_\pi + 2\left\langle Pf, (I-P)^{-1} f\right\rangle_\pi. $$
Then 
$ \lim_{n\rightarrow \infty} \sigma^2_{\mathrm{GM},n} =\sigma_f^2.$
Moreover, we have the following rate of convergence:
$$ \left|\sigma_f^2 -\sigma_{\mathrm{GM},n}^2\right|\leq \frac{C}{4n-6}$$
where  $C$ is a constant depending on $f$ and $P$ but independent of $n$. 
\end{thm}

\begin{figure}[h]
\begin{centering}
\includegraphics[width=0.5\columnwidth]{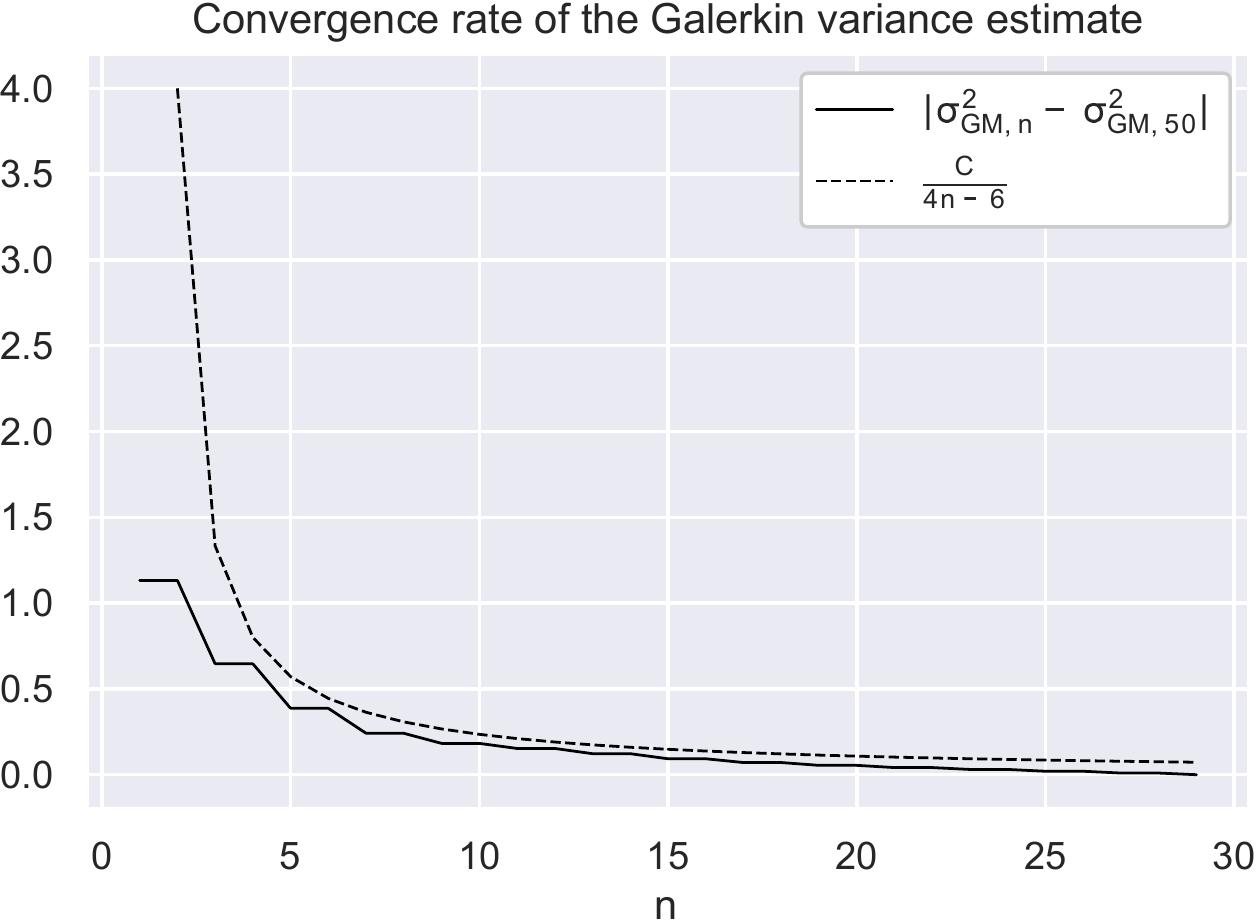}
\par\end{centering}
\caption{Error bound for the Galerkin approximation of $\sigma_f^2$ given in Theorem \ref{thm:galerkin}.}
\label{fig:Galerkin error}
\end{figure}

In practice, the drawback in the Galerkin method arises in finding the entries of the matrix $G$.
We compute the entries of $G$ by introducing a finite dimensional matrix $P_M$ 
that approximates the Markov operator $P$
and perform numerical integration.
We now describe how $P_M$ is constructed.
Given the billiard cell $\mathcal{M}$ with phase space $\mathcal{V}=\mathbb{T}\times\X$, 
we partition $\mathbb{T}$ and $\X$ into $N$ and $M$ 
evenly spaced subintervals $\{I_1\ldots, I_N\}$ and $\{J_1\ldots, J_M\}$, respectively. 
For each subinterval in the partitions, we choose a representative element, 
e.g. the midpoint, to construct the sequences $\{r_k\}_{k=1}^{N}$ and $\{x_\ell\}_{\ell=1}^{M}$, respectively. 
For each pair $(r,x)$ in the set $\{(r_k,x_\ell) \;:\; 1\leq k\leq N, 1\leq \ell\leq M  \}$ 
we then simulate the standard billiard motion of a particle in the cell $\mathcal{M}$ 
with initial conditions $(r,x)$ and record the particle's velocity upon its return to $\mathbb{T}$. 
The finite rank approximation $P_M$ is then the $M\times M$ matrix 
whose $ij$ entry is the proportion of $N$ trajectories 
whose initial velocity $x_i\in J_i$ yield a return velocity in the subinterval $J_j$.

We mention here that $\sigma_{GM,n}^2$ stabilizes for moderately sized values of $M$ (say $M=1000)$,
 independent of the choice of $n$.
Moreover, we have used the matrix $P_M$ to give another numerical approximation of $\sigma_f^2$.
For larger values of $h$, it is possible to solve the Markov-Poisson equation $(P_M-I)g=-f$
using a standard numerical linear system solver (which implements the bi-conjugate stabilized method, or BiC method).
We should mention that the BiC method is used simply to give numerical verification for the use
of the Galerkin method in the large $h$ regime. For small values of $h$, the spectral gap of $P_M$ is small
and the condition number of $I-P_M$, propositional to the inverse of the spectral gap of $P_M$, is too high to be reliable.
In Figure \ref{fig:variances}, we have shown a comparison of approximations for $\sigma_f^2$
for the small bumps family introduced in the previous subsection and the observable $f_a$ defined in (\ref{eqn:truncated})
as produced by the three methods: (1) using equation (\ref{eq:variance-h-approx}) truncated to $n$ terms and denoted by $\sigma^2_{\text{Lser},n}$, (2) using the Galerkin method with dimension $n$ and denoted by $\sigma^2_{\text{GM},n}$,
and (3) using the BiC linear system solver.

\begin{figure}[h]
\begin{centering}
\includegraphics[width=1\columnwidth]{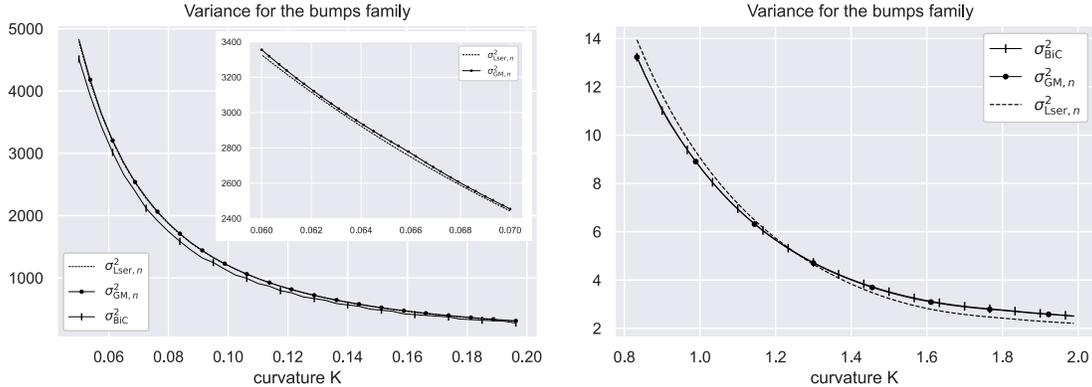}
\par\end{centering}
\caption{Comparison of the variance, computed using the methods described in this section, 
for the simple bumps family with observable $f$ 
given by the horizontal displacement function along the length of the channel with cut-off $a=50000$
(there are no qualitative differences for larger choices of the cut-off). 
The value of $\sigma^2_{\mathrm{Lser},n}$ is computed using (\ref{eq:variance-h-approx}) 
where we have used the first $n = 500$ terms in the series. 
For $\sigma^2_{\mathrm{GM},n}$ we have used dimension $n=200$ and a finite rank approximation $P_M$
of $P$ with $M=1000$.
The same matrix $P_M$ has been used for $\sigma^2_{\text{BiC}}$.
On the left, the dimensionless curvature parameter $K$ is relatively small, while on the right it is relatively large.
The inset of the graph on the left zooms in on the smallest values of $K$, where the 
$\sigma^2_{\mathrm{Lser},n}$ and $\sigma^2_{\mathrm{GM},n}$ approximations are valid.} 
\label{fig:variances}
\end{figure}

We conclude the section with the result of two more numerical experiments.
A first  example is given by the   family of microstructures depicted in Figure \ref{fig:twobumps}. There are two competing curvatures, which are fixed while the height parameter $d$ varies over a range of positive and negative values. When $d<0$, the higher curvature bump is more exposed and when $d>0$
the smaller curvature bump is on top. 

\begin{figure}[h]
\begin{centering}
\includegraphics[width=0.5\columnwidth]{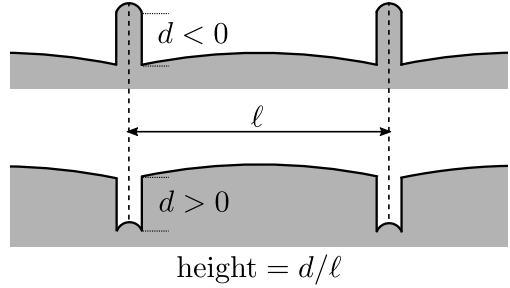}
\par\end{centering}
\caption{The two-bumps family. By varying the parameter $d$ keeping the curvatures constant, we can investigate how the two curvatures compete against each other in the determination of the spectral gap and the dimensionless coefficient of self-diffusivity $\eta$. The result is shown in Figure \ref{fig:twobumps_plot}.}
\label{fig:twobumps}
\end{figure}

The numerical results are shown in the plots of Figure \ref{fig:twobumps_plot}. The interpretation is somewhat straightforward: when the bigger curvature bump is more exposed to collision with the particles, scattering is more diffuse, spectral gap is larger, and diffusivity is smaller (slower diffusion), than
when the less curved bump rises above the other. 
Perhaps more surprising is the near perfect mirror  symmetry between the two graphs. 

\begin{figure}[h]
\begin{centering}
\includegraphics[width=0.6\columnwidth]{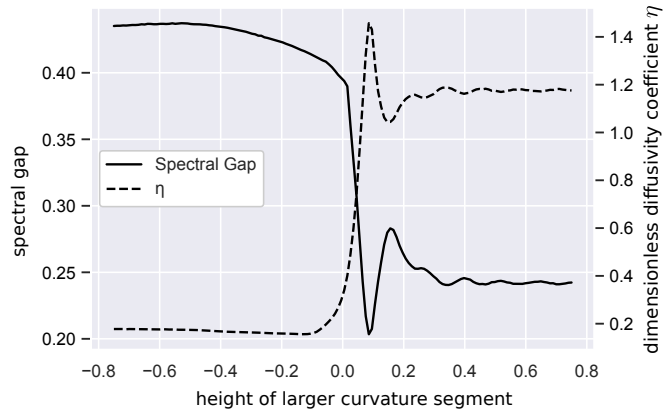}
\par\end{centering}
\caption{Spectral gap and dimensionless coefficient of self-diffusivity for the microstructure of Figure \ref{fig:twobumps}.}
\label{fig:twobumps_plot}
\end{figure}

 In the second example we obtain the dimensionless diffusivity and spectral gap for the 
one-parameter family of microstructures indicated in Figure \ref{fig:bumps_with_wall_figure}. Here the parameter investigated is the (dimensionless) width
of the flat top wall, while the radius $R$ of the curved part is kept constant.  Diffusivity is computed using the Galerkin method (dimension $200$) while the spectral gap is obtained more directly by computing eigenvalues of the finite dimensional approximation of $P$.

\begin{figure}[h]
\begin{centering}
\includegraphics[width=0.4\columnwidth]{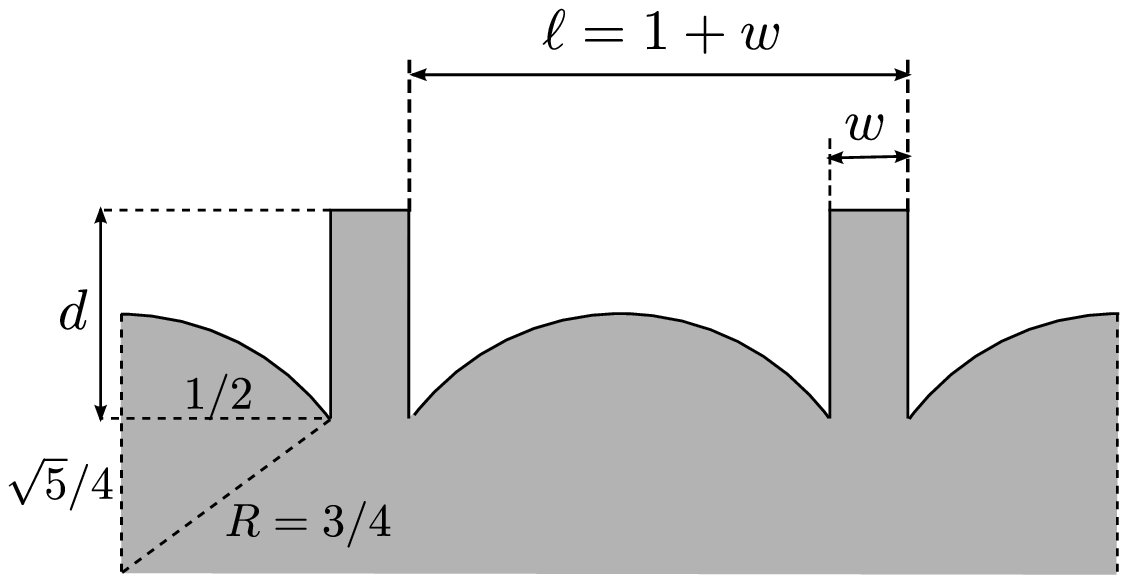}
\par\end{centering}
\caption{Bumps with flat top wall microstructure. The geometric parameters we vary are the relative width $w$ and height $d$ of the flat top wall.  }
\label{fig:bumps_with_wall_figure}
\end{figure}

The results are now somewhat harder to interpret. The interplay between the flat wall top,  the curvature of the middle bumps, and reflection on the sides
of the walls creates a qualitatively more complicated effect. Nevertheless, both this and the previous example show a marked transition in the values of
diffusivity and spectral gap as the height of the wall (with curved top in the first example and flat top in the second) crosses  the height of the adjacent curved segments. Once again, we observe near mirror symmetry in the graphs of spectral gap and diffusivity as functions of the geometric parameter. This is an interesting observation that merits further investigation.

\begin{figure}[h]
\begin{centering}
\includegraphics[width=0.8\columnwidth]{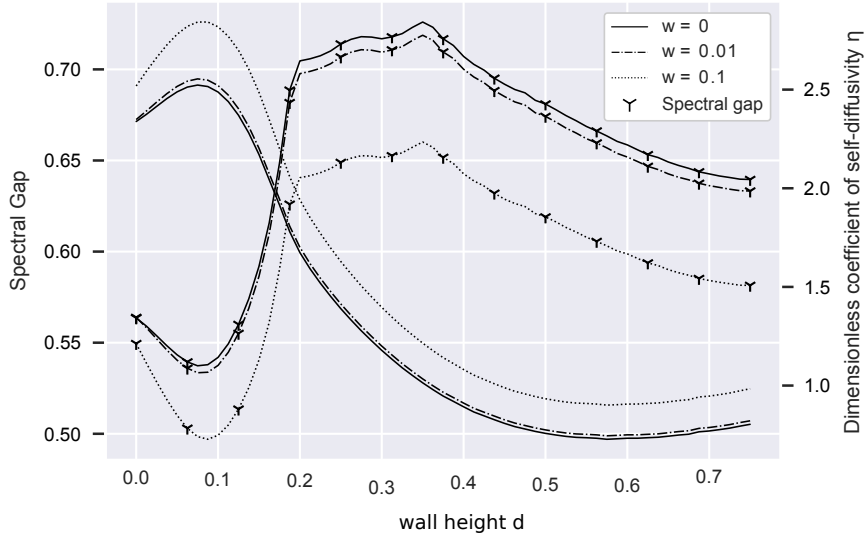}
\par\end{centering}
\caption{Spectral gap and diffusivity coefficient $\eta$ for the microstructure shown in Figure \ref{fig:bumps_with_wall_figure}.
The geometric parameters being varied are the relative height  $d$ of the flat top wall and its relative width $w$. }
\label{fig:bumps_with_wall_plot}
\end{figure}

\section{Spectral gap and ergodicity\label{sec:Spectral-gap}}

The theorems of the previous sections will be strengthened and proved in this and following sections.
We begin this section by introducing a useful technique for decomposing
the operator $P$. The idea will be to condition on the event that
a billiard trajectory within the microscopic cell satisfies certain
properties, which will allow us to focus attention on geometric features
of the microgeometry that create mixing in the dynamics. More specifically,
we show here that under assumptions to be stated, the transition probability
operator $P$ for the random billiard Markov chain has a spectral
gap by showing that for certain components of the decomposition it
is a Hilbert-Schmidt operator. This, along with an additional geometric
assumption that yields a reversible Markov chain, in turn will give
ergodicity.

Let $N:=\mathbb T \times \mathcal X$
be the space of initial conditions of a scattering event and let $N_{1},N_{2},\ldots$
be a measurable partition of $N$. For each $x \in \mathcal X$ and $i\geq 1$, let $N_{i}(x):=\left\{ r\in\mathbb{T}:(r,x)\in N_{i}\right\} $.
Define $\alpha_{i}(x):=|N_{i}(x)|$,
where $|\cdot|$ denotes the size of a set under the normalized Lebesgue measure on $\mathbb T$.
For each $f\in L^{2}(\mathcal{X},\pi)$, define
\[
(P_{i}f)(x)=\begin{cases}
\frac{1}{\alpha_{i}(x)}\int_{N_{i}(x)}f(X(r,x))\,dr & \text{if \ensuremath{\alpha_{i}(x)\neq0}}\\
0 & \text{if \ensuremath{\alpha_{i}(x)=0.}}
\end{cases}
\]
We refer to $P_{i}$ as the \emph{conditional
operator} associated to partition element $N_{i}$. Note that $P_{i}\mathbbm{1}_{A}(x)$ is the
conditional probability that the outgoing velocity vector is in $A\subset\mathcal{X}$
given pre-collision velocity $x$ and given that the event
$N_{i}$ holds. Let $\pi_{i}$ denote the measure on $\mathcal{X}$
such that $\pi_{i}(dx)=\alpha_{i}(x)/(dr\otimes\pi)(N_{i})\,dx$.
Then $\pi_{i}$ is the conditional measure given by $\pi$ conditioned
on the event that $N_{i}$ holds. Finally, observe that for any $f\in L^{2}(\mathcal{X},\pi)$,
it makes sense to decompose $P$ as follows:
\begin{equation}\label{eq:decomposition}
(Pf)(x)=\sum_{i}\alpha_{i}(P_{i}f)(x).
\end{equation}
We now outline some properties of the conditional operators and the
resulting decomposition of $P$. For details of proofs, see \cite{FZ2012}.
\begin{prop}
Let $P_{i}$, $i \geq 1$, be the conditional operators associated to the measurable
partition $N_{1},N_{2},\ldots$ of the space $N$ of initial conditions
of billiard trajectories within billiard microcell $M$, and let $\pi_{j}$
be the conditional measures associated to the partition. Then for each $i \geq 1$,
\begin{enumerate}
\item $P_{i}$ has norm 1.
\item Each term $\alpha_{i}P_{i}$ in the decomposition has norm at most
$\left\Vert \alpha_{i}\right\Vert _{\infty}.$
\item If $N_{i}$ is symmetric---that is, it is invariant under the map
$(r,x)\mapsto(1-r,Jx)$ where $Jx$ denotes the reflection across
the vertical axis in $\mathbb{H}_{-}^{2}$ of the velocity vector corresponding to $x$ and $\mathbb{T}$ is identified
with the unit interval--- then $P_{i}$ is self-adjoint as an operator
on $L^{2}(\mathcal{X},\pi_{i})$.
\end{enumerate}
\end{prop}

The following assumptions will be shown to be sufficient
for ergodicity.
\begin{assumption}
\label{assu:reversibility}The billiard cell is symmetric
with respect to reflection across the vertical axis given by the map
$(x,y)\mapsto(-x,y)$.
\end{assumption}
\begin{assumption}
\label{assu:partition}There exists a measurable partition $N_{1},N_{2},\ldots$ whose elements are symmetric and
such that the following holds for at least one partition element $N_{j}$.
\begin{enumerate}
\item The trajectories with initial conditions in $N_{j}$ collide only
with portions of the boundary of the microscopic billiard cell consisting
piecewise smooth concave curves whose curvatures are bounded below by a constant $K>0$.
\item $\inf_{v\in\mathcal{X}}\alpha_{j}(v)>0$.
\end{enumerate}
\end{assumption}
Note that these assumptions are not optimal---for example,
billiard cells with convex sides have been shown to give
geometrically ergodic random billiard Markov chains in \cite{CFZ2016}---but
capture a large class of examples like those in Section \ref{sec:definitions and results}.
The key idea of Assumption \ref{assu:partition} is that partitioning the phase space
and subsequently decomposing the Markov transition operator into corresponding conditional operators
allows us to focus our study of the operator only on the features that create enough dispersion to yield ergodicity.

\begin{thm}
\label{thm:small-bumps-gap}Let $P$ be the Markov transition operator
for a random billiard Markov chain whose billiard cell
satisfies Assumptions \ref{assu:reversibility} and \ref{assu:partition}.
Then $P$ is a self-adjoint operator with spectral gap. As a result,
there exists a constant $\rho\in(0,1)$ such that for each probability
measure $\mu\in\mathcal{P}(\mathcal{X})$, absolutely continuous with respect to $\pi$ with $\|\mu\|_\pi<\infty$,
there exists a constant $M_{\mu}<\infty$ such that $\left\Vert \mu P^{n}-\pi\right\Vert _{v}\leq M_{\mu}\rho^{n}.$ 
\end{thm}

Note that Theorem \ref{thm:small-bumps-gap} generalizes Theorem \ref{thm:geo_erg}. Indeed, for billiard cells that satisfy the geometric property in the hypotheses of Theorem \ref{thm:geo_erg}, it is clear that for each $x$, there exists an open set $W_x^1 \subset \mathbb{T}$ such that for each $r \in W_x^1$, the billiard trajectory with initial condition $(r,x)$ results in one collision with the boundary of the billiard cell before returning to the reference line. Letting $N_1 = \{(r,x) : x \in \mathcal{X}, r \in W_x^1\}$ and $N_2 = N\setminus N_1$, it is clear that Assumptions \ref{assu:reversibility} and \ref{assu:partition} are satisfied.
We also note that Theorem \ref{thm:small-bumps-gap} includes as a special case, the case of i.i.d. mixtures of microstructures.

The proof of Theorem \ref{thm:small-bumps-gap} requires a series of lemmas, which we now introduce. Note that these lemmas are adapted from a series of lemmas in \cite{Feres2007} but the present statements have more relaxed hypotheses on the geometry of the billiard cell and thus are stronger.

Before stating the first lemma, we need to introduce some notation. Consider a measurable partition satisfying the conditions in Assumption \ref{assu:partition}, where $N_j$ and $P_j$ are the partition element and corresponding conditional operator that satisfy the restrictions in the assumption. Let $W_x^i := \{r \in \mathbb{T} : (r,x) \in N_i\}$ for each partition element $N_i$. For each $x \in \mathcal X$, we let $X_x : \mathbb T \to \mathcal X$ be the function given by $X_x(r) = X(r,x)$, where $X(r,x)$ is the return velocity at the reference line of the billiard cell for a trajectory with initial condition $(r,x)$.

\begin{lem}
\label{lem:kernel-def}Suppose the billiard cell satisfies Assumption \ref{assu:partition}, with partition element $N_j$ and conditional operator $P_j$ satisfying the conditions in the assumption. Then for all $x \in \mathcal X$, the set $W_x^j = \{r \in \mathbb{T} : (r,x) \in N_j\}$ consists of a countable union of open intervals $W_{x,i} \subset \mathbb{T}$. Moreover, the restriction $X_{x,i} := X_x|_{W_{x,i}}$ is a diffeomorphism from $W_{x,i}$ onto its image $V_{x,i}$. Finally, when we use the convention $\mathcal X = (0,\pi)$, we have that for all $f \in L^2(\mathcal{X}, \pi_j)$, $P_jf(x)=\int_{\mathcal{X}}f(\phi)\,\omega(x,\phi)\,\pi_j(d\phi)$
where
\begin{equation}
\omega(x,\phi):= \frac{(\lambda\otimes\pi)(N_j)}{\alpha_j(x)\alpha_j(\phi)}\sum_{i}\mathbbm{1}_{V_{x,i}}(\phi)\left(\frac{1}{2}\left|X_{x,i}'\left(X_{x,i}^{-1}(\phi)\right)\right|\sin\phi\right)^{-1}\label{eq:kernel-1}
\end{equation}
and $\mathbbm{1}_{V_{x,i}}$ denotes the indicator function of
the set $V_{x,i}$. 
\end{lem}

\begin{proof}
We begin by outlining some standard facts in the theory of classical billiards. See \cite{CM2006} for details.
Let $\Gamma$ denote the boundary of the billiard cell $Q$ and note that $\Gamma = \bigcup_i \Gamma_i$ consists of a union of smooth component curves, or walls. 
We denote by $\Gamma_0$ the reference line, which is identified with $\mathbb{T}$. Let $\mathcal{M} = \bigcup_i \mathcal{M}_i$ be the collision space, where each set $\mathcal{M}_i$ consists of pairs $(q,v)$ where $q \in \Gamma_i$ and $v$ points into the interior of $Q$. 
The billiard map $\mathcal{F}:\mathcal{M} \to \mathcal{M}$ is the map defined so that $\mathcal{F}(q,v)$ gives the pair $(q',v')$ where $q'$ is the first intersection of the ray $q+tv$, $t>0$, with $\partial Q$. 
The normalized measure $m\otimes\pi \in \mathcal{P}(\mathcal{M})$, where $m$ is the normalized arclength measure on $\partial Q$, is left invariant by $\mathcal{F}$.
Moreover, if we let $T:N \to N$ be the first return map of billiard orbits, the measure $dr \otimes \pi$, where $dr$ is the normalized Lebesgue measure on $\mathbb{T}$, is left invariant by $T$.
By Poincar\'e recurrence, there is a subset $E_0 \subset N$ of full $dr\otimes\pi$ measure of orbits that start at and return to $\Gamma_0$ in a finite number of steps, and the orbits are non-singular, ie. they do not hit corners of boundary and there are no grazing tangential collisions.
As a result, for each $(q,v)\in E_0$, there is an open neighborhood in $N$ whose elements return to $N$ in the same number of steps as $(q,v)$ and the return map on this set is smooth.
In a similar fashion, it follows that the map $X_x : \mathbb{T} \to \mathcal{X}$ is smooth on an open subset of $\mathbb{T}$ and its restriction to the set $W_x^j$ is likewise a diffeomorphism on an open set which consists of a countable union of open intervals $W_{x,i} \subset \mathbb{T}$. It is also the case that for dispersing billiards, e.g. those billiards for which $\partial Q$ consists of smooth convex curves with positive curvature, the restriction $X_{x,i}$ of $X_x$ to the set $W_{x,i}$ has the property that $X_{x,i}'\neq 0$. Moreover, the summation in (\ref{eq:kernel-1}) is well defined; see \cite[Lemma 5.56]{CM2006}.

We conclude the proof with a verification that the function $\omega$ defined in  (\ref{eq:kernel-1}) is a kernel for $P_j$. Let $A \subset \mathcal{X}$ be a measurable set and let $A_{x,i} = \{r \in W_{x,i} : X_{x,i}(r) \in A\}$. Then
\begin{align*}
\int_{A} \omega(x, \phi)\,\pi_j(d\phi) &= \frac{1}{\alpha_j(x)} \sum_{i} \int_{A \cap V_{x,i}} \left(\frac{1}{2}\left|X_{x,i}'\left(X_{x,i}^{-1}(\phi)\right)\right|\sin\phi\right)^{-1}\, \pi(d\phi) \\
&= \frac{1}{\alpha_j(x)} \sum_{i} \int_{X_{x,i}(A_{x,i})} \left(\left|X_{x,i}'\left(X_{x,i}^{-1}(\phi)\right)\right|\right)^{-1}\, d\phi \\
&= \frac{1}{\alpha_j(x)} \sum_{i} \int_{A_{x,i}} dr \\
&= P_j \mathbbm{1}_{A}(x).
\end{align*}
Since this relation holds for indicator functions, it follows by a standard argument using linearity and the density of simple functions in $L^2(\mathcal X, \pi_j)$ that $P_j$ has kernel $\omega$ for all $f \in L^2(\mathcal{X},\pi_j)$.
\end{proof}

The next intermediary lemma gives an estimate on the kernel in Lemma \ref{lem:kernel-def}. Its proof follows from \cite[Lemmas 6.5, 6.6, 6.7]{Feres2007} with only minor modifications.

\begin{lem}
\label{lem:kernel-estimates}
Consider a billiard cell satisfying Assumption \ref{assu:partition} and let $\omega$ be the kernel given in (\ref{eq:kernel-1}). Then $\omega \in L^2(\mathcal{X}\times\mathcal{X}, \pi_j\otimes\pi_j)$.
\end{lem}

The following lemma is adapted from \cite[Theorem 9.9]{W1980}. It will be used to show that for an operator $P$ which admits a decomposition as in (\ref{eq:decomposition}), it suffices to show that one conditional operator is compact in order to prove that $P$ has spectral gap. The notation $\|\cdot\|$ is used to denote the canonical Hilbert space operator norm.

\begin{lem}\label{lem:gap}
Let $K$ and $T$ be bounded self-adjoint operators on a Hilbert space
and suppose that $K$ is compact. Then the essential spectrum of $T+K$
is contained in the essential spectrum of $T$. In particular, if
$\|T+K\| =1$ and $\|T\| <1$,
then the spectral gap of $T+K$ satisfies $\gamma(T+K)\geq\min\left\{ 1-\|T\| ,\gamma(K)\right\} $.
\end{lem}

We conclude with the proof of the section's main theorem.

\begin{proof}[Proof of Theorem \ref{thm:small-bumps-gap}]
That $P$ is self adjoint follows from Assumption \ref{assu:reversibility} and Proposition \ref{prop:P}. To see that $P$ has spectral gap, we apply Lemma \ref{lem:gap}. Using the notation of the lemma, we let $K = \alpha_jP_j$ and $T = \sum_{i \neq j}\alpha_iP_i$. Then, applying Lemmas \ref{lem:kernel-def} and \ref{lem:kernel-estimates}, we have that $K$ is a Hilbert-Schmidt integral operator and hence it is compact. It is clear that $T$ is bounded and self-adjoint and $\|T+K\| = \|P\| = 1$, where $\|\cdot\|$ is the $L^2$-operator norm. Moreover, $1 - \|T\| \geq \inf_{v\in\mathcal{X}} \alpha_j(v) > 0$. It follows that the spectral gap of $P$ is strictly positive. The concluding statement of exponential convergence to the stationary measure in total variation then follows immediately using the Cauchy-Schwarz inequality since $\|\mu\|_v \leq 1/2\|\mu\|_{\pi}$.
\end{proof}

\section{Diffusivity\label{sec:Diffusivity}}

Let $f:\mathcal{X}\to\mathbb{R}$ be a function on the state space
of the random billiard Markov chain $(X_{n})_{n\geq0}$ with Markov
transition operator $P$. We refer to $f$ as an \emph{observable}
(or \emph{functional})\emph{ }of the Markov chain. Without loss of
generality, we suppose that it has mean zero with respect to the stationary
distribution: $\pi(f)=0$. Our focus in this section will be on the
limiting distribution (after appropriate scaling) of partial sums
of the functional of the Markov chain given by
\[
S_{n}(f):=\sum_{k=0}^{n-1}f(X_{k}).
\]
It is well known that under appropriate mixing conditions for the
Markov chain, $S_{n}(f)/\sqrt{n}$ converges in distribution to a
centered Gaussian distribution with variance parameter $\sigma_{f}^{2}$.
As a preliminary result, we show that random billiard Markov chains
with microstructure have sufficiently fast mixing for a (central)
limit theorem of this kind to hold. However, our primary focus will
be to show that the variance $\sigma_{f}^{2}$ of the limiting Gaussian
distribution, which we refer to as the \emph{diffusivity} of the system,
can be rigorously approximated, and formulas can be derived in terms
of geometric parameters for families of random billiard microstructures.

We use here a result adapted from \cite{KV1986}, which states that
the central limit theorem holds for reversible Markov chains satisfying
a nondegeneracy condition on $\sigma_{f}^{2}$.
\begin{thm}
\label{thm:KV1986}Let $(X_{n})_{n\geq0}$ be a Markov chain with
stationary measure $\pi$ and let $f\in L_0^2(\mathcal X, \pi)$. If the
Markov chain is reversible, then $S_{n}(f)/\sqrt{n}$ converges in
distrubtion to a centered Gaussian random variable $\mathcal{N}(0,\sigma_{f}^{2})$
as long as $\sigma_{f}^{2}<\infty$, where $\sigma_f^2$ is given by \eqref{eq:variance-correlations}.
\end{thm}
In the discussion that follows, it will be useful to express $\sigma_{f}^{2}$
in terms of the spectrum of $P$, viewed as an operator on $L^{2}(\mathcal{X},\pi)$.
We first note that since $P$ is a bounded, self-adjoint operator
on $L^{2}(\mathcal{X},\pi)$ with norm 1, there exists a projection-valued
measure $\Pi$, supported on the spectrum $\sigma(P)\subset[-1,1]$
of $P$, defined so that
\[
P=\int_{-1}^{1}\lambda\,\Pi(d\lambda).
\]
For each $f\in L_{0}^{2}(\mathcal{X},\pi)$, we further define a measure
$\Pi_{f}$ supported on $\sigma(P)\setminus\left\{ 1\right\} $ by
$\Pi_{f}(d\lambda):=\left\langle f,\Pi(d\lambda)f\right\rangle _{\pi}.$
Now, observe that
\begin{align}
\sigma_{f}^{2} & =\left\langle f,f\right\rangle _{\pi}+2\sum_{k=1}^{\infty}\left\langle f,P^{k}f\right\rangle _{\pi}\label{eq:variance-correlations} \\
 & =\left\langle f,f\right\rangle _{\pi}+2\left\langle f,P(I-P)^{-1}f\right\rangle _{\pi}\label{eq:variance-Poisson}\\
 & =\int_{-1}^{1}\frac{1+\lambda}{1-\lambda}\Pi_{f}(d\lambda).\label{eq:variance}
\end{align}
Using the expression in (\ref{eq:variance}), we show that the existence
of a positive spectral gap is sufficient for the central limit theorem
to hold.
\begin{cor}
\label{cor:KV1986-RR1997}Let $(X_{n})_{n\geq0}$ be a Markov chain
with Markov transition operator $P$ and stationary measure $\pi$.
Let $f \in L_0^2(\mathcal X, \pi)$. If the Markov chain is reversible and $P$
has spectral gap $\gamma>0$, then $S_{n}(f)/\sqrt{n}$ converges
in distrubtion to a centered Gaussian random variable $\mathcal{N}(0,\sigma_{f}^{2})$.
\end{cor}
\begin{proof}
Since $P$ has spectral gap, there exists $0<\rho<1$ such that for
every $\lambda\in\text{supp\,}(\Pi_{f})$, $\lambda\leq\rho$. Therefore,
$\sigma_{f}^{2}$, as given by (\ref{eq:variance}), is finite
since $\sigma_{f}^{2}\leq(1+\rho)/(1-\rho)\pi(f^{2})<\infty.$
\end{proof}

\subsection{Diffusion approximation and diffusivity\label{subsec:Diffusion-approximation}}

We now prove Theorem \ref{thm:diffusion-app}.
Recall that the boundary of the billiard cell is assumed to be the graph of a periodic function $F:\mathbb{T} \rightarrow \mathbb{R}$.
Also recall the definitions of $h$ and $\mathcal{L}$ from Subsection 2.4.

\begin{proof}[Proof of Theorem \ref{thm:diffusion-app}]
When only a single boundary surface collision occurs, the relationship
between the initial and return velocity vectors, $v=(x, v_0)$ and $V(r,v)$ respectively, is straightforward. Indeed, let $\mathbbm n:\mathbb T \to \mathbb R^2$ denote the vector field of normal vectors along the boundary of the billiard cell and let $\bar{\mathbbm n}$ and $\mathbbm n_0$ denote the first (horizontal) and second (vertical) components of $\mathbbm n$. If collision with the boundary surface occurs at the point $(r', F(r'))$, 
then $V(r,v)=v-2\left\langle v,\mathbbm n(r')\right\rangle \mathbbm n(r')$, where $\langle \cdot, \cdot \rangle$ denotes the Euclidean inner product. Note
that by elementary geometry
\[
\mathbbm n(r')=\frac{1}{\sqrt{1+F'(r')^{2}}}\left(-F'(r'),1\right)^\intercal,\quad r=r'-\left(F(r')-c\right)x/v_{0}.
\]
It now follows that for any smooth function $f:(-1,1)\to\mathbb{R}$
\begin{align*}
Pf(x) & =\int_{\mathbb{T}}f\left(x-2\left\langle v,\mathbbm n(r')\right\rangle \bar{\mathbbm n}(r')\right)\,dr\\
 & =\int_{\mathbb{T}}f(x-2(\alpha+\beta)\bar{\mathbbm n}(r'))\left(1+\alpha/\beta\right)\,dr',
\end{align*}
where $\alpha=\bar{\mathbbm n}x,\beta=\mathbbm n_{0}v_{0}.$ Moreover, by Assumption \ref{assu:reversibility}, the symmetry relations
$\bar{\mathbbm n}(\ell-r)=-\bar{\mathbbm n}(r)$ and $\mathbbm n_{0}(\ell-r)=\mathbbm n_{0}(r)$ hold. Using these relations, and suppressing the explicit dependence of $\bar{\mathbbm n}$ on $r'$ for the sake of simplicity of notation, and we get
\[
Pf(x)=\frac{1}{2}\int_{\mathbb{T}}\left[f(x-2(\alpha+\beta)\bar{\mathbbm n})\left(1+\alpha/\beta\right)+f(x+2(-\alpha+\beta)\bar{\mathbbm n})\left(1-\alpha/\beta\right)\right]\,dr'.
\]
From here we use the second order Taylor approximation of $\phi$
centered about $x$. Observe that for $w\in(-1,1)$
\[
f(x+w)=f(x)+f'(x)w+\frac{f''(x)}{2}w^{2}+R_{x}(w),
\]
where $R$ is the usual Taylor remain term $R_{x}(w)=f'''(c)w^{3}/3!$
for some $c$ in the interval between $x$ and $w$. Using this,
together with straightforward algebraic manipulation that we omit
for the sake of clarity of exposition, we get that
\begin{align*}
Pf(x) & =f(x)-4xf'(x)\int_{\mathbb{T}}\bar{\mathbbm n}^{2}\,dr'+f''(x)\int_{\mathbb{T}}\bar{\mathbbm n}^{2}(6\alpha^{2}+2\beta^{2})\,dr'+E(x)\\
 & =f(x)-4xf'(x)h+2\left(1-x^{2}\right)f''(x)h+O(h^{2})+E(x)\\
 & =f(x)+2h\frac{d}{dx}\left(\left(1-x^{2}\right)f'(x)\right)+O(h^{2})+E(x),
\end{align*}
where $h=\int_{\mathbb{T}}\bar{\mathbbm n}^{2}\,dr'$ and $E$ is an error term. The error term
arises from the remainder $R$ and is bounded as follows: $|E|\leq C_{\phi}p(x, v_0)I_{3}$,
where $C_{\phi}$ is a constant that depends only on the third derivative
of $\phi$, $p(x,v_{0})$ is a polynomial in $x,v_{0}$
of degree at most 3 with coefficients that do not depend on $\phi$,
and $I_{3}:=\int_{\mathbb{T}}\bar{\mathbbm n}^{3}\,dr'$. 
\end{proof}

\subsection{Computing the diffusivity}
The differential operator $\mathcal{L}$ defined in (\ref{eq:legendre}) has a well understood spectral
theory. We will take advantage of this in the following to give a method for computing $\sigma_f^2$. 
Before going on, we first note a few well known facts about $\mathcal{L}$.

\begin{prop}
Let \emph{$\mathcal{L}$ be the Legendre differential operator defined
in (\ref{eq:legendre}). The following properties
hold.}
\begin{enumerate}
\item The eigenvalue problem $\mathcal{L}f=\lambda f$ has solutions
if and only if $\lambda$ is of the form $\lambda=-l(l+1)$ for integers
$l\geq0$. 
\item The solutions of the eigenvalue problem are the polynomials $\phi_{l}$,
$l\geq0$, known as the Legendre polynomials. The first few are given
by $\phi_{0}=1,\phi_{1}(x)=x,\phi_{2}(x)=(3x^{2}-1)/2$. 
\item The collection $(\phi_{l})_{l\geq0}$ of Legendre polynomials form
a complete orthogonal basis for $L^{2}(\mathcal X, \pi)$ and
\[
\langle\phi_{n},\phi_{m}\rangle:=\int_\mathcal X\phi_{n}(x)\phi_{m}(x)\,\pi(dx)=\frac{1}{2n+1}\delta_{n,m},
\]
where $\delta_{n,m}$ is the Kronecker delta symbol. 
\end{enumerate}
\end{prop}

We are now ready to discuss the diffusivity $\sigma_{f}^{2}$ introduced
at the start of the section. The idea will be to use the diffusion
approximation $\mathcal{L}$ of the Markov-Laplacian $L$ in order
to give an approximation of the function $g=(I-P)^{-1}f$ that arises
in (\ref{eq:variance-Poisson}). Note that $g$ is a solution
of the Markov-Poisson equation $Lg=-f$. We first show that a series
solution of the classical Poisson equation can be given explicitly
in terms of Legendre polynomials.
\begin{lem}
\label{lem:series-solution}For any $f\in L_{0}^{2}(\mathcal X, \pi)$, the equation
$\mathcal{L}g=-f$ has solution given by
\[
g=\sum_{l=1}^{\infty}a_{l}\phi_{l},\quad a_{l}=\frac{2l+1}{2l(l+1)}\left\langle \phi_{l},f\right\rangle_\pi .
\]
\end{lem}
\begin{proof}
Let $f\in L_{0}^{2}(\mathcal X, \pi)$. Since the Legendre functions form a complete
orthogonal basis for $L_{0}^{2}(\mathcal{X},\pi)$, $f=\sum_{l=1}^{\infty}b_{l}\phi_{l}$,
where $b_{l}=(2l+1)\left\langle \phi_{l},f\right\rangle _{\pi}$.
Now, let $g=\sum_{l=1}^{\infty}a_{l}\phi_{l}$, where $a_{l}=b_{l}/(l(l+1))$.
Observe that
$\mathcal{L}g =\sum_{l=1}^{\infty}a_{l}\mathcal{L}\phi_{l}=-\sum_{l=1}^{\infty}a_{l}l(l+1)\phi_{l}=-f.
$
\end{proof}
With the lemma in hand, we now give our main approximation result.
The idea of the proof will be to contruct a series solution approximation
of the Markov-Poisson equation using the series solution of the Poisson
equation along with the diffusion approximation of $P$. We use the
estimates in Theorem \ref{thm:diffusion-app} to control
the error terms in our approximation.
\begin{thm}
Let $(P_{h})_{h>0}$ be a family of random billiard Markov transition
operators for a family of microscopic billiard cells satisfying Assumptions
 \ref{assu:reversibility} and \ref{assu:partition}. For any function $f\in L_{0}^{2}(\mathcal X, \pi)$,
let $\sigma_{f,h}^{2}$ denote the diffusivity corresponding to $P_{h}$.
Then
\begin{equation}
\sigma_{f,h}^{2}=-\langle f,f\rangle_\pi+\frac{1}{h}\sum_{l=1}^{\infty}\frac{2l+1}{l(l+1)}\left\langle \phi_{l},f\right\rangle_\pi ^{2}+O(h^{1/2}).\label{eq:variance-h-approx}
\end{equation}
\end{thm}
\begin{proof}
Let $h>0$ and let $g_{h}$ be the solution of the Poisson equation
$\mathcal{L}g=-f/(2h)$. Note that by Lemma \ref{lem:series-solution},
$g_{h}=\sum_{l=1}^{\infty}a_{l,h}\phi_{l}$ where
\[
a_{l,h}=\frac{2l+1}{2hl(l+1)}\left\langle \phi_{l},f\right\rangle_\pi .
\]
By Theorem \ref{thm:diffusion-app}, $Lg_{h}=2h\mathcal{L}g_{h}+O(h^{1/2})=-f+O(h^{1/2}).$
Note that the error in the above expression is of lower order than
that in the theorem because the right hand side in the Poisson equation
contains a factor of $h^{-1}$. Next observe that
\begin{align*}
\left\langle Pf,g_{h}\right\rangle_\pi  & =\left\langle f,Pg_{h}\right\rangle_\pi \\
 & =\left\langle f,g_{h}\right\rangle_\pi +2h\left\langle f,\mathcal{L}g_{h}\right\rangle_\pi +O(h^{1/2})\\
 & =\frac{1}{2h}\sum_{l=1}^{\infty}\frac{2l+1}{l(l+1)}\left\langle \phi_{l},f\right\rangle_\pi ^{2}-\left\langle f,f\right\rangle_\pi +O(h^{1/2}).
\end{align*}
Using the expression above, along with the formula for $\sigma_{f,h}^{2}$
given in (\ref{eq:variance-Poisson}), the result then follows.
\end{proof}

\section{Analysis of the Galerkin method\label{sec:Numerical-techniques}}

In this section, we conclude with an analysis of the Galerkin method introduced in Subsection \ref{sec:Numerical-techniques-examples}, including a proof of Theorem  \ref{thm:galerkin}.
We begin with a result on the decay rates of Legendre series truncation which will be useful. It is taken from Theorem 2.2 from \cite{wang2018new}, restated slightly here to fit our notation and context.

We will give estimates in terms of the following weighted semi-norm, defined on the space of functions $u : (-1,1) \to \mathbb R$ such that the following integral is defined:

$$ \left\Vert u\right\Vert_w =  \int_{-1}^{1}\frac{|u'(x)|}{(1-x^2)^{\frac{1}{4}}} \,dx.$$

\begin{thm}[Adapted from Theorem 2.2 in \cite{wang2018new}]\label{thm:legendre_decay}
Let $m \geq 1$, and let $u : (-1,1) \to \mathbb R$ be a function such that $u, u',\ldots, u^{(m-1)}$ are absolutely continuous and the $m$-th derivative $ u^{(m)}$ is of bounded variation. Furthermore, assume that $\left\Vert u^{(m)} \right\Vert_w < \infty$. Let $a_n = (2n+1) \langle u, \phi_n \rangle$ be the sequence of coefficients in the Legendre expansion of $u$ such that $u(x) = \sum_{n = 0}^\infty a_n \phi_n(x)$.  Then, for $n \geq m+1$,
$$|a_n| \leq \frac{\left\Vert u^{(m)}\right\Vert_w}{\sqrt{\pi(2n - 2m - 1)}} \prod_{k=1}^{m} \frac{2}{2n - 2k + 1}.$$
\end{thm}

We are now ready to prove Theorem \ref{thm:galerkin}. We begin by recalling some notation introduced in Subsection \ref{sec:Numerical-techniques-examples}. Let $T_n:L_0^2(\mathcal{X},\pi) \to R_n$ denote the orthogonal projection to the linear span $R_n=\{\phi_1, \dots, \phi_n\}$ of the first $n$ non-constant Legendre polynomials. The solution of the finite dimensional linear system $ (I-T_nP)x = T_nf$, where $f \in L_0^2(\mathcal X, \pi)$ is the given observable, will be written as $g_n$. Note that $g_n \in R_n$, and writing $g_n=\sum_{j=1}^n \alpha_j \phi_j$, it is straightforward to see that we aim to find a solution $x$ to the system $Gx = y$ where 
$$x=(\alpha_1, \dots, \alpha_n)^\intercal, \ \ y= (\langle f, \phi_1\rangle_\pi, \dots, \langle f, \phi_n\rangle_\pi)^\intercal, \ \ G=\left(\langle \phi_j,\phi_i\rangle_\pi - \langle P\phi_j, \phi_i\rangle_\pi\right)_{i,j=1}^n.$$
The solution $g_n$ will in turn be used to give the approximation $\sigma^2_{\mathrm{GM},n} := \langle T_n f, T_n f\rangle + 2\langle T_nPf, g_n\rangle$ of the diffusivity $\sigma_f^2$.

\begin{proof}[Proof of Theorem \ref{thm:galerkin}]
Throughout the proof, we take $P$ to be the restriction of the Markov operator
to the space $L_0^2(\mathcal X, \pi)$ so that $\|P\| < 1$, where $\|\cdot\|$ denotes $L^2$-operator norm.
Observe that from the definition of $\sigma^2_{\mathrm{GM},n}$,
\begin{equation}\label{eq:gm-diff}
\left|\sigma_f^2 - \sigma^2_{\mathrm{GM},n}\right| \leq \left \langle f, f\rangle_\pi - \langle T_nf, T_nf\rangle_\pi \right|
	+ 2\left| \langle Pf, g \rangle_\pi - \langle T_n Pf, g_n \rangle_\pi \right|.
\end{equation}
Our aim is to show that the two terms of the right hand side above are bounded by a common factor in terms of $n$,
which we will then show decays as in the statement of the theorem.
For the first term on the right hand side of (\ref{eq:gm-diff}), we see that
	\begin{align}
		\inner{f}{f}_{\pi} - \inner{T_nf}{T_nf}_{\pi}
		& = \|f\|_\pi^2 - \|T_nf\|_\pi^2 \nonumber \\ 
		& \leq 2\|f\|_\pi (\|f\|_\pi -\|T_nf\|_\pi ) \nonumber\\
		& \leq 2\|f\|_\pi \|f - T_nf\|_\pi \label{eq:term1},
	\end{align}
and for the second term,
	\begin{align*}
		\left|\inner{Pf}{g}_{\pi}-\inner{T_nPf}{g_n}_{\pi}\right|
		&= \left|\int_\mathcal{X} [Pf(x)g(x) - T_nPf(x)g_n(x) ]\,\pi(dx)\right|\\
		&\leq \int_\mathcal{X} |Pf(x)(g(x)-g_n(x))|\,\pi(dx) + \int_\mathcal{X}|g_n(x)(Pf(x) -T_nPf(x))|\,\pi(dx)\\
		&\leq \|Pf\|_\pi\|g-g_n\|_\pi + \|g_n\|_\pi \|Pf-T_nPf\|_\pi\\
		&\leq \|Pf\|_\pi \|(I-T_nP)^{-1} \| \|g-T_n g \|_\pi + \|g_n\|_\pi \|Pf-T_nPf\|_\pi,
	\end{align*}
	where in the last step we have used the fact that $g - g_n = (I-T_nP)^{-1}(g - T_n g)$.
	
It is straightforward to see that $T_n P = PT_n$. 
Indeed, for any function $f \in L_0^2 (\mathcal X, \pi)$, with Legendre expansion given by $f = \sum_{k=1}^\infty a_k \phi_k$, we have that $Pf(x) = \sum_{k = 1}^\infty a_k P\phi(x).$
Applying $T_n$ to both sides of this equation, we get that
$T_nPf(x) = \sum_{k=1}^n a_k P\phi(x) = PT_nf(x).$
It now follows that $$\|Pf-T_nPf\|_\pi = \|Pf-PT_nf\|_\pi \leq \|P\| \|f-T_nf\|_\pi.$$
Moreover, a similar argument gives that 
$$\|g-T_ng\|_\pi \leq \|(I-P)^{-1}\| \|f-T_nf\|_\pi \leq (1-\|P\|)^{-1}\|f-T_nf\|_\pi.$$
Finally, we note that
$$\|(I-T_nP)^{-1} \|
	\leq (1-\|T_nP\|)^{-1}
	\leq (1-\| P\|)^{-1},$$
and consequently,
$$g_n = (I-T_nP)^{-1}T_nf \leq (1-\|P\|)^{-1}\|f\|_\pi.$$
We then have for the second term on the right hand side of (\ref{eq:gm-diff}),
\begin{equation}\label{eq:term2}
\left|\inner{Pf}{g}_{\pi}-\inner{T_nPf}{g_n}_{\pi}\right| \leq  
	\|P\|(2-\|P\|)(1-\|P\|)^{-2}\|f\|_\pi\|f-T_nf\|_\pi.
\end{equation}
Applying the estimates in (\ref{eq:term1}) and (\ref{eq:term2}) to (\ref{eq:gm-diff}) and simplifying, we have that
\begin{equation}\label{eq:last-estimate}
\left|\sigma_f^2 - \sigma^2_{\mathrm{GM},n}\right| \leq 2\|f\|_\pi(1-\|P\|)^{-2}\|f - T_n f\|_\pi.
\end{equation}
It is now evident that the convergence rate will depend on the decay rate of $f$ with it's Legendre series truncation.
Observe that
$$\|f-T_nf\|_\pi \leq \sum_{k=n+1}^\infty a_k \|\phi_k\|_\pi = \sum_{k=n+1}^\infty (2k+1)^{-1/2}a_k,$$
where $a_k = (2k+1)|\langle f, \phi_k\rangle_\pi|$. Using Theorem \ref{thm:legendre_decay} with $m=1$ we get
	\begin{equation*}
		|a_k| \leq \frac{2\|f'\|_{w}}{\sqrt{\pi}(2k-1)\sqrt{2k-3}}. 
	\end{equation*}
	Thus
	$$\|f-T_nf\|_\pi
	\leq \frac{2s_n\| f'\|_{\mathrm{w}}}{\sqrt{\pi}},$$
	where
	$$s_n:=\sum_{k=n+1}^{\infty} \frac{1}{(2k-1)\sqrt{2k+1}\sqrt{2k-3}}.$$
Further, for $n>2$ we have
	\begin{align*}
		s_n &< \sum_{i=n+1}^{\infty} \frac{1}{(2i-3)^{2}} \\
		&\leq  \int_n^{\infty} \frac{dx}{(2x-3)^{2}}\\
		&=\frac{1}{4n-6}. 
	\end{align*}	
The result now follows by applying these estimates to (\ref{eq:last-estimate}).		
\end{proof}

\end{document}